\documentclass[12pt, reqno]{amsart}
\usepackage{amssymb, amscd, amsthm, mathtools}
\usepackage{url,amssymb,enumerate,colonequals}
\usepackage[margin = 1in]{geometry}
\usepackage{pdfsync}
\usepackage{multirow}
\usepackage[usenames,dvipsnames]{xcolor}
\usepackage{url}
\usepackage{microtype}
\usepackage{tikz}
\usepackage[all]{xy}
\usetikzlibrary{matrix}

\newcommand{\C}{\mathbb{C}}
\newcommand{\F}{\mathbb{F}}
\newcommand{\Q}{\mathbb{Q}}
\newcommand{\Z}{\mathbb{Z}}
\newcommand{\Fbar}{{\overline{\F}}}
\newcommand{\calG}{\mathcal{G}}
\newcommand{\calO}{\mathcal{O}}

\newcommand{\cl}{\mathrm{cl}}
\newcommand{\fp}{\mathfrak{p}}
\newcommand{\fP}{\mathfrak{P}}
\newcommand{\fq}{\mathfrak{q}}
\newcommand{\fD}{\mathfrak{D}}
\newcommand{\fQ}{\mathfrak{Q}}

\DeclareMathOperator{\End}{\mathrm{End}}
\DeclareMathOperator{\Gal}{\mathrm{Gal}}
\DeclareMathOperator{\Trd}{\mathrm{Trd}}
\DeclareMathOperator{\Nrd}{\mathrm{Nrd}}
\DeclareMathOperator{\Aut}{\mathrm{Aut}}

\DeclareMathOperator{\Fpbar}{\overline{\F}_p}
\DeclareMathOperator{\Ell}{\mathrm{Ell}}

\DeclareMathOperator{\id}{\mathrm{Id}}
\DeclareMathOperator{\M}{\mathrm{M}}
\DeclareMathOperator{\SE}{\mathrm{SE}}
\DeclareMathOperator{\XE}{\mathrm{XE}}
\DeclareMathOperator{\SSj}{\mathrm{SSj}}

\newtheorem{theorem}{Theorem}[section]
\newtheorem{lemma}[theorem]{Lemma}

\newtheorem{proposition}[theorem]{Proposition}

\theoremstyle{definition}
\newtheorem{definition}[theorem]{Definition}

\newtheorem{example}[theorem]{Example}

\newtheorem{remark}[theorem]{Remark}

\numberwithin{equation}{subsection}

\begin{document}
\title[Endomorphism rings of supersingular elliptic curves over $\F_p$]{Endomorphism rings of supersingular elliptic curves over $\F_p$}

\author{Songsong Li, Yi Ouyang, Zheng Xu}
\address{Wu Wen-Tsun Key Laboratory of Mathematics,  School of Mathematical Sciences, University of Science and Technology of China, Hefei, Anhui 230026, China}
\email{songsli@mail.ustc.edu.cn}
\email{yiouyang@ustc.edu.cn}
\email{xuzheng1@mail.ustc.edu.cn}

\maketitle	
\begin{abstract}
Let $p>3$ be a fixed prime. For a supersingular elliptic curve $E$ over $\F_p$ with $j$-invariant $j(E)\in \F_p\backslash\{0, 1728\}$, it is well known that  the Frobenius map $\pi=((x,y)\mapsto (x^p, y^p))\in \End(E)$ satisfies ${\pi}^2=-p$. A result of Ibukiyama tells us that $\End(E)$ is a maximal order in $\End(E)\otimes\Q$ associated to a (minimal) prime $q$ satisfying $q\equiv 3 \bmod 8$ and the quadratic residue $\bigl(\frac{p}{q}\bigr)=-1$ according to $\frac{1+\pi}{2}\notin \End(E)$ or  $\frac{1+\pi}{2}\in \End(E)$. Let $q_j$ denote the minimal $q$ for $E$ with $j=j(E)$. Firstly, we determine the neighborhood of the vertex $[E]$ in the supersingular $\ell$-isogeny graph if $\frac{1+\pi}{2}\notin \End(E)$ and $p>q\ell^2$  or  $\frac{1+\pi}{2}\in \End(E)$ and $p>4q\ell^2$. In particular, under our assumption, we show that there are at most two vertices defined over $\F_p$ adjacent to $[E]$. Next, under GRH, we obtain the bound $M(p)$ of $q_j$ for all $j$ and estimate the number of supersingular elliptic curves with $q_j<c\sqrt{p}$. We also compute the upper bound $M(p)$  for all $p<2000$ numerically and show that $M(p)>\sqrt{p}$ except $p=11,23$ and $M(p)<p\log^2 p$ for all $p$.
\end{abstract}

\section{Introduction}\label{sec: introduction}	

We fix a prime $p>3$.  Computing Super-Singular Isogeny Problem (CSSI), i.e. computing  an isogeny between two supersingular elliptic curves defined over $\Fpbar$, is supposed to be a hard computational problem, Jao and De Feo~\cite{DFJLP14} designed a key agreement scheme, and S.~Galbraith, C.~Petit and J.~ Silva ~\cite{GPS2016} proposed an identification scheme and a signature scheme based on the hardness of this problem.

Let $\ell\neq p$ be another fixed prime. One way to construct isogenies between supersingular elliptic curves is using the supersingular $\ell$-isogeny graph $\calG_{\ell}(\Fpbar)$. Recall $\calG_{\ell}(\Fpbar)$ is a directed graph, whose set of vertices  $V_{\ell}(\Fpbar)$ are $\Fpbar$-isomorphism classes of supersingular elliptic curve $[E]$ defined over $\Fpbar$ and whose edges are equivalent classes of $\ell$-isogenies defined over $\Fpbar$ between two elliptic curves in the isomorphism classes. As usual  the vertices are represented by $j$-invariants.  It is well-known  (see~\cite{AEC}) that every supersingular elliptic curve over $\Fpbar$ has $j$-invariant in $\F_{p^2}$, thus $V_{\ell}(\Fpbar)=V_{\ell}(\F_{p^2})$ and further investigation tells us that its cardinality is $\lfloor\frac{p}{12}\rfloor + \varepsilon$ where $\varepsilon= 0$, $1$ or $2$ depending on the class of $p\mod{12}$. As seen in~\cite{Kohel}, $\calG_{\ell}(\Fpbar)$ is an expander graph whose diameter is $O(\log p)$, thus there is a short path between any two vertices. Finding paths in $\calG_{\ell}(\Fpbar)$ then leads to constructing isogenies between two supersingular elliptic curves. Numerous work recently is based on this idea (see~\cite{Pohl},\cite{GS13}).

We shall study the supersingular elliptic curves over $\F_p$, i.e. the $F_p$-vertices of the graph  $\calG_{\ell}(\Fpbar)$ in this note. It is well known (see \cite{GD16} or \cite[Theorem 14.18]{Cox89})
 \[ \# \{j\in \F_p\mid \text{$j$ is a supersingular invariant}\}=\begin{cases} \frac{1}{2}h(-p), &\ \text{if}\ p\equiv 1\bmod 4,\\ h(-p), &\ \text{if}\ p\equiv 7\bmod 8,\\ 2h(-p), &\ \text{if}\ p\equiv 3\bmod 8,\end{cases} \]
where $h(-p)$ is the class number of $\Q(\sqrt{-p})$. When $p\rightarrow\infty$, by the Brauer-Siegel Theorem (\cite[Theorem 4.9.15]{Coh96}), $h(-p)$ is approximately $\sqrt{p}$ or $2\sqrt{p}$ if $p\equiv 3$ or $1\bmod{4}$. If $j\in \F_p$ is a supersingular $j$-invariant, we pick one  supersingular elliptic curve $E_j$ over $\F_p$ with $j(E_j)=j$. In a previous work \cite{YZ}, we determined the neighborhood of $[E_j]$ in $\calG_{\ell}(\Fpbar)$ for $j=1728$ if $p>4\ell^2$  and $j=0$ if $p>3\ell^2$. In this note, we shall work on the endomorphism rings of supersingular elliptic curves  with $j$-invariants in  $\F_p\backslash\{0,1728\}$.

For a supersingular elliptic curve $E$ over $\F_{p^2}$, its endomorphism ring $\End(E)$ is a maximal order in the unique definite quaternion algebra $B_{p,\infty}$ over $\Q$ ramified only at $p$ and $\infty$ (see \cite{Voight}).
Furthermore $j(E)\in \F_p$ if and only if  $\End(E)$ contains a root of $x^2+p=0$. If $j\in \F_p\backslash \{0, 1728\}$ is a supersingular $j$-invariant, let $\pi=((x,y)\mapsto (x^p, y^p))$ be the absolute Frobenius in $\End(E_j)$, it can be shown that $\pm \pi$ are the only roots of $x^2+p$ in $\End(E_j)$.

For $q$ a prime satisfying $q\equiv 3 \bmod 8$ and  the quadratic residue $\bigl(\frac{p}{q}\bigr)=-1$, let  $H(-q,-p)=\Q\langle 1,i,j,k\rangle$ be the quaternion algebra over $\Q$ defined by $i^2=-q$, $j^2=-p$ and $ij=-ji=k$. By computing the discriminant of $H(-q,-p)$ one sees that $B_{p,\infty}\cong H(-q,-p)$. We identify these two quaternion algebras by the isomorphism.  Let
 \[ \calO(q)  :=\Z\langle1, \frac{1+i}{2}, \frac{j+k}{2}, \frac{ri-k}{q}\rangle\ \text{where}\ r^2+p\equiv 0\bmod{q},\]
and allowing also $q=1$,
 \[ \calO'(q):=\Z\langle1, \frac{1+j}{2}, i, \frac{r'i-k}{2q}\rangle\ \text{where}\ p\equiv 3\bmod{4}, \ r^{\prime 2}+p\equiv 0\bmod{4q}. \]
 Then  $\calO(q)$ and  $\calO'(q)$ are maximal orders in $B_{p,\infty}$. Note that the choices of $r$ and $r'$ in $\Z$ are not essential, up to isomorphism the orders $\calO(q)$ and  $\calO'(q)$ depend only on $q$ (and of course $p$). Then for $j\in  \F_p$ a supersingular $j$-invariant, Ibukiyama~\cite{Ibukiyama} showed that $\End(E_j)$ is isomorphic to $\calO(q)$ if $\frac{1+\pi}{2}\notin \End(E_j)$ (equivalently,   $\End(E_j)\cap \Q(\pi)=\Z[\pi]$) or $\calO'(q)$ if  $\frac{1+\pi}{2}\in \End(E_j)$ (equivalently,  $\End(E_j)\cap \Q(\pi)=\Z[\frac{1+\pi}{2}]$) for some $q$. In particular, $\End(E_0)\cong \calO(3)$ and $\End(E_{1728})\cong \calO'(1)$.

However, $q$ is not unique. Let $q_j$ be minimal  such that $\End(E_j)\cong \calO(q_j)$ or $\calO'(q_j)$. Certainly $q_0=3$ and $q_{1728}=1$. When $q_j$ is small comparing to $p$, we can apply the techniques in our previous work \cite{YZ} to determine the neighborhood of $[E_j]$ in the supersingular isogeny graph.
Let $H_D(x) \in \Z[x]$ be the Hilbert class polynomial of an imaginary quadratic order with discriminant $D$.  Define
\[{\delta}_{D} = \begin{cases} 1, & \text{if}\ \big(\frac{D}{\ell}\big)=1 \ \text{and}\ H_D(x) \ \text{splits  into  linear  factors  in} \ \F_{\ell}[x];\\ -1, & \text{otherwise}. \end{cases}\]
We have

\begin{theorem}~\label{theo:ww}
Let $j\in \F_p\backslash\{0,1728\}$ be a supersingular $j$-invariant and $\pi$ be the Frobenius map of $E_j$. Suppose $\ell\nmid 2pq_j$.
\begin{itemize}
\item[(i)] In the case $\frac{1+\pi}{2}\notin \End(E_j)$, i.e. $\End(E_j)=\calO(q_j)$, if $p > q_j\ell^2$, there are $1+{\delta}_{-q_j}$ loops of $[E_j]$ and $\ell-{\delta}_{-q_j}$ vertices adjacent to $[E_j]$ in $\calG_{\ell}(\Fpbar)$ and hence each connecting to $[E_j]$ by one edge.
\item[(ii)] In the case $\frac{1+\pi}{2}\in \End(E_j)$, i.e.  $\End(E_j)=\calO'(q_j)$, if $p > 4q_j\ell^2$, there are $1+{\delta}_{-4q_j}$ loops of $[E_j]$ and $\ell-{\delta}_{-4q_j}$ vertices adjacent to $[E_j]$ in $\calG_{\ell}(\Fpbar)$ and hence each connecting to $[E_j]$ by one edge.
\end{itemize}
In both cases, there are $1+\big(\frac{-p}{\ell}\big)$ vertices defined over $\F_p$ adjacent to $[E]$ with an $\F_p$-edge.
\end{theorem}

Unfortunately, numerical evidence tells us that $q_j$ might be larger than $p$.  Let $M(p)=\max\{q_{j}\mid j\ \text{is a supersingular invariant over}\ \F_p\}$. In the appendix we collect  data of $M(p)$ for $p<2000$, which  reveals that $M(p)>\sqrt{p}$ except $p=11$ and $M(p)<p\log p$ for all $p$. Under Generalized Riemann Hypothesis (GRH), we obtain the following result.

\begin{theorem} \label{theo:bound} Let $p>3$ be a prime. Assume GRH (Generalized Riemann Hypothesis) holds.

$(1)$ For any constant $C>0$, if $p$ is sufficiently large, there exists a supersingular invariant $j$ such that $q_j>C\sqrt{p}$.

$(2)$ For a generic  supersingular $j$-invariant $j\in \F_p\backslash\{0,1728\}$,  $q_j< 10000 p\log^{4}p$.

$(3)$ For any  supersingular $j$-invariant $j\in \F_p\backslash\{0,1728\}$,  $q_j< 10000 p\log^{6}p$.

$(4)$ Let $N(x)=\#\{q_j\leq x\mid\ j\ \text{is a supersingular $j$-invariant\ in\ $\F_p$}\}$. Then
\begin{itemize}
\item[(i)] If $p \equiv1 \bmod 4$, then $N(4\sqrt{p}) \sim \frac{\sqrt{p}}{\log p}$ as $p\rightarrow \infty$.
\item[(ii)] If $p \equiv3 \bmod 4$, then $N(\frac{\sqrt{p}}{2}) \sim \frac{\sqrt{p}}{4\log p}$ as $p\rightarrow \infty$ and $\liminf N(4\sqrt{p})\frac{\log p}{\sqrt{p}}\geq \frac{9}{8}$.

\end{itemize}
\end{theorem}


\subsection*{Acknowledgments} Research is partially supported by Anhui Initiative in Quantum Information Technologies (Grant No. AHY150200) and NSFC (Grant No. 11571328). Y. O. would like to thank the Morningside Center of Mathematics for hospitality where part of this paper was written.


\section{Preliminaries}

\subsection{Elliptic curves over finite fields}\label{sec: elliptic curves}
In this subsection, we introduce some basic knowledge about elliptic curves over finite fields, one can refer to ~\cite{AEC} for details. Let $\F$ be a finite field of characteristic $p>3$, let $\bar{\F}$ be the algebraic closure of $\F$. An elliptic curve $E$ defined over $\F$ is a projective curve with affine model $E: y^2=x^3+Ax+B$ where $A,B \in \F$ and $4A^3+27B^2 \neq 0$. The $j$-invariant of $E$ is $j(E)=1728\frac{4A^3}{4A^3+27B^2}$. The set of $\F$-rational points on $E$ is $E(\F)=\{(x,y) \in {\F}^2: y^2=x^3+Ax+B\}\cup \{\infty\}$, where $\infty$ is the point at infinity. Then $E(\F)$ is a finite abelian group.

Let $E$ and $E'$ be two elliptic curves defined over $\F$. An isogeny $\phi: E\rightarrow E' $ is a morphism satisfying $\phi(\infty)=\infty$. If $\phi(E)=\{\infty\}$, we say $\phi=0$. If $\phi\neq 0$, then $\phi$ is a surjective group homomorphism with finite kernel, and  we call $E$ and $E'$ are isogenous. The isogeny $\phi$ is called an $L$-isogeny if it is defined over $L$ (i.e. written as rational maps over $L$), $\phi$ is called separable (resp. inseparable) if the corresponding field extension $\Fbar(E)/\phi^* \Fbar(E')$ is  separable (resp. inseparable). The degree of $\phi$ is the degree of the field extension $\Fbar(E)/\phi^* \Fbar(E')$. If $\phi$ is separable, in particular if $p\nmid \deg\phi$, then $\deg(\phi)=\#\ker(\phi)$.  If $\deg(\phi)=1$, $E$ and $E'$ are isomorphic. Particularly, if $j(E)=j(E')$, then $E$ and $E'$ are isomorphic over $\Fbar$.

An endomorphism of $E$ is an isogeny from $E$ to itself. The set $\End(E)$ of all endomorphisms of $E$ form a ring under the usual addition and composition as multiplication. As in ~\cite{AEC}, $\End(E)$ is either an order in an imaginary quadratic extension of $\Q$ or a maximal order in a quaternion algebra over $\Q$. In the first case $E$ is called ordinary, in the second case $E$ is called supersingular. Moreover, every supersingular elliptic curve over $\Fbar_p$ is isomorphic to an elliptic curve defined over $\F_{p^2}$. Consequently, we may and will assume the supersingular elliptic curve $E$ we study is  defined over $\F_{p^2}$.

\subsection{Number theoretic background}
In this subsection, we introduce some basic knowledge in number theory needed later. Most of them can be found in \cite{Neuk99, Cox89, Sut17}. We shall use big $O$ to denote an order in a number field and  calligraphic $\calO$ to denote an order in a quaternion algebra over $\Q$ as in \S~2.3.

For $M$ a number field, let $O_M$, $I_M$, $P_M$ and $h_M$ be the ring of integers, the group of fractional ideals, the group of principal ideals and the class number of $M$.

Let $M/N$ be an extension of number fields of degree $[M:N]=m$. Then $O_M$ is free $O_N$-module of rank $m$. Let $\{e_1,\cdots, e_m\}$ be a basis of $O_M$ over $O_N$ and  $\{\sigma_1,\cdots, \sigma_m\}$ be the set of $N$-embeddings of $M$ in an algebraic closure $\overline{\Q}$ of $\Q$, then the discriminant $D_{M/N}:=(\det (\sigma_i(e_j))_{i,j})^2\in O_N$. Let $O_M^*$ be the dual $O_N$-module of $O_M$ under the trace map, then the different $\fD_{M/N}$ is the inverse of $O_M^*$, which is an ideal of $O_M$. We write $D_{M/\Q}=D_M$.
\begin{proposition} \label{prop:disc} Suppose $M/N$ is an extension of number fields. Then
 \[ N_{M/N}(\fD_{M/N})=D_{M/N}. \]
where $N_{M/N}: M\rightarrow N$ is the norm map. Moreover,
 \begin{itemize}\item[(i)] If $L$ is an intermediate field in $M/N$, then \[ \fD_{M/N}=\fD_{M/L}\cdot \fD_{L/N},\quad
  D_{M/N}=(D_{L/N})^{[M:L]}\cdot N_{L/N}(D_{M/L}). \]

 \item[(ii)] Let $M_1$ and $M_2$ be number fields, $N=M_1\cap M_2$ and $M=M_1 M_2$. Suppose $M_1$ and $M_1$ are linearly disjoint over $N$, i.e. $[M:N]=[M_1:N]\cdot [M_2:N]$. Then
     \[  \fD_{M/N}\mid \fD_{M_1/N} \fD_{M_2/N},  \quad
  D_{M/N}\ \Big{|}\ D^{[M:M_1]}_{M_1/N}\cdot D_{M_2/N}^{[M:M_2]}.  \]
 If $D_{M_1/N}$ and $D_{M_2/N}$ are moreover coprime, then
     \[ \fD_{M/N}= \fD_{M_1/N} \fD_{M_2/N},\quad D_{M/N}=(D_{M_1/N})^{[M_2:N]} \cdot (D_{M_2/N})^{[M_1:N]}. \]
 \end{itemize}
\end{proposition}
\begin{proof} All are standard facts, except the first part of (ii), which we prove here for lack of reference. By (i), $\fD_{M/N}=\fD_{M/M_1}\fD_{M_1/N}$. By assumption, an $N$-embedding $\sigma: M_2\hookrightarrow \overline{\Q}$ extends uniquely to an $M_1$-embedding $M\hookrightarrow \overline{\Q}$. If $\{e_1,\cdots, e_n\}$ is a basis of $O_{M_2}$ over $O_N$, let $R$ be the $O_{M_1}$-submodule of $O_M$ generated by $\{e_1,\cdots, e_n\}$. By definition, under the trace map of $M/M_1$, $R^*$ is  $(\fD_{M_2/N}O_M)^{-1}$, $O_M^*$ is  $\fD_{M/M_1}^{-1}$,  hence we have $\fD_{M/M_1}\mid \fD_{M_2/N}$.
\end{proof}

For a Galois extension $M/N$ of number fields, let $\fp$ be  a prime ideal of $O_N$ and $\mathfrak{P}$ a prime of $O_M$ lying above $\mathfrak{p}$. Suppose $\fP/\fp$ is unramified. The Frobenius automorphism $\big[\frac{M/N}{\mathfrak{P}}\big]$ is the unique element $\sigma \in G=\Gal(M/N)$ such that
 \[\sigma(\alpha)= {\alpha}^{N(\mathfrak{p})} \mod \mathfrak{P}\ \text{for\ all}\ \alpha \in O_M \]
where $N(\fp)=\# (O_N/\fp)$.  All $\big[\frac{M/N}{\mathfrak{P}}\big]$, when $\fP$ varies over primes above $\fp$, form a conjugate class in $\Gal(M/N)$, which we denote by $\big[\frac{M/N}{\mathfrak{p}}\big]$. In the special case that $M/N$ is an abelian extension, $\big[\frac{M/N}{\mathfrak{p}}\big]=\big[\frac{M/N}{\mathfrak{P}}\big]$ is a one-point-set.

For $C$  a conjugacy class in $G$, define the function
 \begin{equation} \pi_C(x, M/N):=\#\{\fp\mid \fp \ \text{is unramified\ in}\ M,\ \big[\frac{M/N}{\mathfrak{p}}\big]=C,\   N(\fp)\leq x\}. \end{equation}
We have the following  explicit Chebotarev density theorem:

\begin{theorem} \label{them:Cheb} For any conjugacy class $C$ of $G=\Gal(M/N)$, the set of primes $\fp$ in $N$ such that $[\frac{M/N}{\fp}]=C$ is  of density $\frac{|C|}{|G|}$, i.e.,
\[ \pi_C(x,M/N) \sim \frac{|C|}{|G|} \frac{x}{\log(x)}. \]
More explicitly, let $n_M=[M:\Q]$ and $d_M=|D_{M}|$, then under GRH, one has
 \begin{equation} \label{eq:cheb} \Big{|}\frac{|G|}{|C|}\pi_C(x,M/N)-\int_2^x\frac{dt}{\log t}\Big{|}\leq \sqrt{x}\left[\Bigl(\frac{1}{\pi}+\frac{3}{\log x}\Bigr)\log d_M+\Bigl(\frac{\log x}{8\pi}+\frac{1}{4\pi}+\frac{6}{\log x}\Bigr)n_M\right].\end{equation}
\end{theorem}
\begin{proof} The first part can be found in any advanced number theory textbook. The explicit formula in the second part is a recent result in \cite{GGM19}.
\end{proof}

Let $K$ be an imaginary quadratic field.  Let $O$ be an order of $K$.  The conductor of $O$ is $f=[O_K:O]$, and the discriminant of $O$ is $D(O)=f^2D_K$. In general $O$ may not be a Dedekind domain if $f>1$, however for any $O$-ideal $\mathfrak{a}$ prime to $f$, $\mathfrak{a}$ has a unique decomposition as a product of prime $O$-ideals which are prime to $f$ (see \cite[Proposition 7.20]{Cox89}).

Let $I(O)$ be the group of proper fractional $O$-ideals prime to $f$ and $P(O)$ be the group of principal fractional $O$-ideals prime to $f$, then the ideal class group of $O$ is $\cl(O)=I(O)/P(O)$ and the ideal class number of $O$ is $h(O)=\# \cl(O)$.
Let $I_K(f)$ be the  group of  fractional $O_K$-ideals prime to $f$ and $P_K(f)$ be the principal ideals in $I_K(f)$. Let $P_{K,\Z}(f)$ be the group of principal ideals in $P_K(f)$ generated by $x$ with $x\equiv n\bmod{fO_K}$ for $n\in \Z$ (and relatively prime to $f$). The group $\cl(O)$ is canonically isomorphic to the ring class group $I_K(f)/P_{K,\Z}(f)$. The ring class field $L$ of $O$ is the (unique) abelian extension of $K$ associated by the existence theorem of class field theory to the ring class group of $O$.
The Artin map $\sigma: \cl(O)\cong \Gal(L/K)$ is the canonical isomorphism sending the class of $\fp$ to the Frobenius automorphism $[\frac{L/K}{\fp}]$. Moreover, the uniqueness implies that $L$ is Galois over $\Q$.

For a lattice  $\Lambda \subseteq \C$, let $E_{\Lambda}$ be the elliptic curve over $\C$ such that $E_\Lambda(\C) \cong \C/\Lambda$. Then $E_{\Lambda}\cong E_{\Lambda'}$ (i.e. $j(E_{\Lambda})=j(E_{\Lambda'})$) if and only if $\Lambda=\lambda\Lambda$ for some $\lambda \in \C^{\times}$ (i.e. $\Lambda$ and $\Lambda'$ are homothetic). For $O$ an order in an imaginary quadratic field $K$, let
 \[ \Ell_{O}(\C):= \{j(E)\mid \End(E)\cong O\}\ \ (= \{E\mid \End(E)\cong O\}/\sim). \]
Then $\Ell_{O}(\C)=\{j(E_{\mathfrak{b}})\mid [\mathfrak{b}]\in \cl(O)\}$ and  $\cl(O)$ acts transitively on  $\Ell_{O}(\C)$ by $[\mathfrak{a}]j(E_{\mathfrak{b}})=j(E_{\mathfrak{a}^{-1}\mathfrak{b}})$
(see \cite[Chapter 18]{Sut17}).  On the other hand  the Galois group $\Gal(L/K)$ acts naturally on $\Ell_O(\C)$. These two actions  are compatible with the canonical isomorphism $\sigma: \cl(O)\cong \Gal(L/K)$ (see \cite[Theorem 22.1]{Sut17}).

Now suppose $O$ is of discriminant $D$.
The Hilbert class polynomial $H_D(x)$ is defined as
 \[H_D(x):=\prod_{j(E) \in \Ell_{O}(\C)}(x-j(E)).\]
From ~\cite{ATAEC}, $H_D(x) \in \Z[x]$. The splitting field of $H_D(x)$ over $K$ is exactly the ring class field $L$ of $O$. One has the following theorem (\cite[Theorem 22.5]{Sut17}):

\begin{theorem}\label{theo:K}
Let $O$ be an imaginary quadratic order of discriminant $D$ and $L$ its ring class field. Let $\ell \nmid D$ be an odd prime which is unramified in $L$. Then the followings are equivalent:
\begin{itemize}
\item[(i)] $\ell$ is the norm of a principal $O$-ideal.
\item[(ii)] The Legendre symbol $\big(\frac{D}{\ell}\big)=1$ and $H_D(x)$ splits into linear factors in $\F_{\ell}[x]$.
\item[(iii)] $\ell$ splits completely  in $L$.
\item[(iv)] $4\ell=t^2-v^2D$ for some integers $t$ and $v$ with $\ell \nmid t$.
\end{itemize}
\end{theorem}

\subsection{Quaternion algebras and maximal orders}\label{sec: Quaternion algebras}
Recall that a definite quaternion algebra over $\Q$ is of the form
 \[ H(-a,-b)= \Q\langle1,i,j,k\rangle,\ i^2=-a,\ j^2=-b,\ k=ij=-ji \]
for some positive integers $a$ and $b$. A lattice in $H(a,b)$ is a $\mathbb{Z}$-submodule of $H(-a,-b)$ of rank $4$ containing a basis of $H(-a,-b)$. There is a canonical involution on $H(a,b)$ defined as
\[\alpha=x+yi+zj+wk \mapsto \bar{\alpha}=x-yi-zj-wk,\  \text{for all} \ \alpha \in H_{a,b}.\]
The reduced trace of $\alpha$ is $\Trd(\alpha)=\alpha+\bar{\alpha}=2x$ and the reduced norm of $\alpha$ is $\Nrd(\alpha)=\alpha\bar{\alpha}=x^2+ay^2+bz^2+abw^2$.

Let $B_{p,\infty}=H(-1,-p)$ be  the unique quaternion algebra over $\Q$ ramified only at $p$ and $\infty$. However, one must keep in mind that there are many pairs of $(a, b)$ such that $B_{p,\infty}=H(-a,-b)$, but
the involution and hence the reduced trace and norm of  $\alpha\in B_{p,\infty}$ are independent of the choice of $(a,b)$.

An order $\mathcal{O}$ in $B_{p,\infty}$ is a lattice which is also a subring of $B_{p,\infty}$. The order $\mathcal{O}$ is called maximal if it is not properly contained in any other order.
For two orders $\mathcal{O}$ and $\mathcal{O}'$ of $B_{p,\infty}$,
we say that they are isomorphic if there exists $\mu \in B_{p,\infty}^\times$ such that $\mathcal{O}'=\mu\mathcal{O}\mu^{-1}$.

For a sublattice $I \subseteq B_{p,\infty}$, we define the left order of $I$ by $\mathcal{O}_{L}(I)=\{x \in B_{p,\infty} \mid xI  \subseteq I\}$ and the  right order of $I$ by $\mathcal{O}_{R}(I)=\{x \in B_{p,\infty} \mid Ix  \subseteq I\}$. If $\mathcal{O}$ is a maximal order and $I$ is a left ideal of $\mathcal{O}$, then $\mathcal{O}_{L}(I)=\mathcal{O}$ and $\mathcal{O}_{R}(I)$ is also a maximal order. For $I$  a left ideal of $\mathcal{O}$, define the reduced norm of $I$ by
\[\Nrd(I)=\gcd\{\Nrd(\alpha) |\ \alpha \in I\}=\sqrt{\mathcal{O}/I}, \]
and define the conjugation ideal of $I$ by $\bar{I}=\{\bar{\alpha} \mid \alpha \in I\}$. Then $\Nrd(\bar{I})=\Nrd(I)$ and \[I\bar{I}=\Nrd(I)\mathcal{O}=\Nrd(\bar{I})\mathcal{O}_{R}(\bar{I}).\]

\subsection{Deuring's correspondence}
Let $E$ be a supersingular elliptic curve over $\F_{p^2}$. From ~\cite{Voight}, $\End(E)=\mathcal{O}$ is a maximal order in $B_{p,\infty}=\End(E)\otimes \Q$. 
For $I$ a left ideal of $\calO$, let $E[I]=\{P \in E(\Fpbar) \mid \alpha(P)=\infty \ \text{for all} \ \alpha \in I\}$, then the quotient map
\[\phi_I: E \rightarrow E_I= E/{E[I]}\]
is an isogeny with $\deg(\phi_I)=\Nrd(I)$.  On the other hand, if $\phi: E\rightarrow E'$ is an isogeny of degree $N$, then $\ker\phi$ is of order $N$ and $I_\phi=\{\alpha\in \calO\mid \alpha(P)=\infty\ \text{for all}\ P\in \ker\varphi\}$ is a left $\calO$-ideal of reduced norm $N$, and there exists an isomorphism $\psi: E_{I_\phi}\cong E'$ such that $\phi=\psi\circ\phi_I$.
Then the following results of Deuring hold (see~\cite[Chapter 42]{Voight}).

\begin{theorem} \label{theo:V}
 Let $E$ be a supersingular elliptic curve over $\F_{p^2}$, and $\End(E)=\mathcal{O}$. Then $\calO$ is a maximal order (up to isomorphism) in $B_{p,\infty}$.
\begin{itemize}
\item[(i)] There is a $1$-to-$1$ correspondence between left ideals $I$ of $\calO$ of reduced norm $N$ and equivalent classes of isogenies $\phi: E\rightarrow E'$ of degree $N$ given by $I\mapsto [\phi_I]$ and $[\phi]\mapsto I_\phi$.

\item[(ii)] If $\phi:E \rightarrow E'$ and $I$ are corresponding to each other, then  $\End(E')\cong \mathcal {O}_R(I)$ is the right order of $I$ in $B_{p,\infty}$. In particular, $\phi\in \End(E)$ if and only if $I=I_\phi=\calO \phi$ is principal.

\item[(iii)] Suppose $\phi_I: E \rightarrow E_I$ and $\phi_J: E \rightarrow E_J$ are isogenies corresponding to the left ideals $I$ and $J$ of $\calO$ respectively. Then $E_I \cong E_J$ if and only if $I$ and $J$ are in the same left class of $\calO$, i.e., $J=I\mu$ for some $\mu \in B_{p,\infty}^\times$.
\end{itemize}
\end{theorem}

Conversely, from \cite[Lemma 42.4.1]{Voight}, let $\calO$ be a maximal order in $B_{p,\infty}$, then $\calO \cong \End(E)$ for some supersingular elliptic curve $E$ over $\F_{p^2}$. More precisely, we have
\begin{lemma}\label{lemma:OtoE}
 Let $\calO$ be a maximal order in $B_{p,\infty}$. Then there exist one or two supersingular elliptic curves $E$ up to isomorphism over $\Fpbar$ such that $\End(E)\cong \calO$. There exist two such elliptic curves if and only if $j(E)\in \F_{p^2}\backslash\F_p$.
\end{lemma}

\begin{lemma} \label{lemma:unique} Suppose $E$ is a supersingular elliptic curve over $\F_{p^2}$. Then $E$ is defined over $\F_p$ if only if that $\End(E)$ contains an element with minimal polynomial $x^2+p$.
Moreover, if $j(E)\neq 0, 1728$, then the absolute Frobenius $\pi=((x,y)\mapsto (x^p,y^p))\in \End(E)$ is the only isogeny up to a sign satisfying $x^2+p=0$.
\end{lemma}
\begin{proof} The equivalence follows from \cite[Proposition 2.4]{GD16}.

Suppose that $\phi \in \End(E)$ satisfying $\phi^2=[-p]$. Then $\hat{\phi}=-\phi$ and $\hat{\phi}\circ \phi=[p]$. Since $E$ is supersingular, $E[p]=\{\infty\}$, thus $\ker \phi =\{\infty\}$ and $\phi$ is inseparable. From \cite[Corollary 2.12]{AEC}, $\phi=\lambda\circ\pi$, where $\lambda \in \End(E)$. Then $\deg(\lambda)=1$. From \cite[Corollary 2.4.1]{AEC}, $\lambda \in \Aut(E)=\{\pm 1\}$ when $j(E) \neq 0$ or $1728$. Thus $\phi=\pm \pi$.
\end{proof}

Ibukiyama~\cite{Ibukiyama} has given an explicit description of all maximal orders $\calO$ in $B_{p,\infty}$ containing a root $\epsilon$ of $x^2+p=0$. Regard $\calO$ and $\Q(\epsilon)$ as subsets in $B_{p,\infty}$, then  $\Q(\epsilon)\cap \calO$ is either $\Z[\epsilon]\cong \Z[\sqrt{-p}]$ or $\Z[\frac{1+\epsilon}{2}]\cong \Z[\frac{1+\sqrt{-p}}{2}]$ where in the latter case $p\equiv 3\bmod{4}$.
Let $q$ be a prime such that
\begin{equation}\label{eq:1}
\big(\frac{p}{q}\big)=-1, \quad q\equiv 3 \mod 8.
\end{equation}
Then the definite quaternion algebra $H(-q,-p)=\Q\langle1,i,j,k\rangle$ with $i^2=-q$, $ j^2=-p$ and $k=ij=-ji$ is also ramified only at $p$ and $\infty$ and we can identify it with $B_{p,\infty}$. By \eqref{eq:1}, $\big(\frac{-p}{q}\big)=1$. Let $r$ be an integer such that $r^2+p \equiv 0 \mod q$ and
 \[ \calO(q):=\Z\langle1, \frac{1+i}{2},\frac{j-k}{2},\frac{ri-k}{q}\rangle. \]
If $p\equiv 3 \mod 4$ and we allow $q=1$, let $r'$ be an integer such that ${r'}^2+p \equiv 0 \mod 4q$ and
 \[ \calO'(q):=\Z\langle1, \frac{1+j}{2},i,\frac{r'i-k}{2q}\rangle. \]
Then $\calO(q)$ and $\calO'(q)$  are maximal orders in $B_{p,\infty}$ which are independent of the choices of $r$ and $r'$ up to isomorphism. From ~\cite{Ibukiyama}, we have

\begin{theorem}\label{THEO:I}
Assume that $\calO$ is a maximal order in $B_{p,\infty}$ containing an element $\epsilon$ with minimal polynomial $x^2+p$. Then there exists a prime  $q$  satisfying condition~\eqref{eq:1} such that $\calO\cong \calO(q)$ if $\calO\cap \Q(\epsilon)=\Z[\epsilon]$  and $\calO\cong \calO'(q)$ or $\calO'(1)$ if $\calO\cap \Q(\epsilon)=\Z[\frac{1+\epsilon}{2}]$  (hence $p \equiv 3 \mod 4$).
\end{theorem}

\begin{remark}
Given a maximal order $\calO$ in the form of $\calO(q)$ or $\calO'(q)$ in $B_{p,\infty}$, by Lemma~\ref{lemma:OtoE}, $\calO$ corresponds to a supersingular ellptic curve $E$ over $\F_p$ such that $\calO\cong \End(E)$. Chevyrev and  Galbraith~\cite{CG14} proposed an algorithm to compute this supersingular elliptic curve with running time $O(p^{1+\epsilon})$.
\end{remark}

Let $j\in \F_p$ be a supersingular $j$-invariant and $E_j$ be the corresponding  supersingular elliptic curve defined over $\F_p$. Then $\End(E_0)\cong \calO(3)$ and $\End(E_{1728})\cong \calO'(1)$. If $j\neq 0, 1728$, then by Theorem~\ref{THEO:I} and Lemma~\ref{lemma:unique}, $\End(E_j)\cong \calO(q)$ if $\frac{1+\pi}{2}\notin \End(E_j)$ and $\End(E_j)\cong \calO'(q)$ if $\frac{1+\pi}{2}\in \End(E_j)$ for some  $q$ satisfying \eqref{eq:1}, and we can identify $\pi$ and $\pm j$ under this isomorphism.  However, $q$ is not unique.
By Lemma 1.8 and Proposition 2.1 of \cite{Ibukiyama}, one has

\begin{lemma} \label{lem:isomophic}
Suppose $q_1\neq q_2$ are primes satisfying \eqref{eq:1}. Let $K=\Q(j)\cong \Q(\sqrt{-p})$. Suppose $q_1$ and $q_2$ have prime decompositions $q_1 O_K=\mathfrak{q}_1\bar{\mathfrak{q}}_1$ and $q_2 O_K=\mathfrak{q}_2\bar{\mathfrak{q}}_2$.

\begin{itemize}

\item[(i)] $\calO(q_1)\cong \calO'(q_2)$ if and only if $|\calO(q_1)^{\times}|=|\calO'(q_2)^{\times}|=4$. Then $\calO(q_1)\not\cong \calO'(q_2)$ if one of them is isomorphic to $\End(E)$ for $j(E)\neq 1728$.

\item[(ii)] $\calO(q_1)\cong \calO(q_2)$ $\Leftrightarrow$  the equation $x^2+4py^2=q_1 q_2$ is solvable over  $\Z$ $\Leftrightarrow$ either $\mathfrak{q}_1\mathfrak{q}_2\in P_{K,\Z}(2)\ \text{or}\ \mathfrak{q}_1\bar{\mathfrak{q}}_2 \in P_{K,\Z}(2)$;

\item[(iii)] $\calO'(q_1)\cong \calO'(q_2)$  $\Leftrightarrow$   the equation $x^2+py^2=4q_1 q_2$ is solvable over  $\Z$  $\Leftrightarrow$ either $\mathfrak{q}_1\mathfrak{q}_2\in P_{K}(2)\ \text{or}\ \mathfrak{q}_1\bar{\mathfrak{q}}_2\in P_{K}(2))$.
\end{itemize}
\end{lemma}

\begin{definition} ~\label{def:bound B}
For $j\in \F_p$ a supersingular $j$-invariant, set
 \[ q_j:=\min\{q\mid  \End(E_j)\cong \calO(q)\ \text{or}\ \calO'(q)\}. \]
Set
 \[ M(p)=\max\{q_j\mid j\ \text{is a supersingular $j$-invariant over}\ \F_p\}. \]
\end{definition}
Certainly $q_0=3$ and $q_{1728}=1$. We shall give the values of $M(p)$ for all primes $p<2000$ in the appendix.

\begin{example} Let $p=101$. We have the following $q_j$:
 \[ q_{57}=11,\ q_{59}=59,\ q_{66}=67,\]
 \[ q_{64}=83,\ q_2=139,\ q_{21}=163. \]
Thus $q_j$ can be bigger than $p$.
\end{example}

\section{Neighborhood of supersingular elliptic curves}\label{sec: neigh of ssec}

In this section we assume that
 \begin{quote} $j\in \F_p\backslash\{0,1728\}$ is a supersingular $j$-invariant, $E=E_j$, $\End(E)=\calO$ and $q=q_j$.  \end{quote}
In this case, then
 \[ \pi=\pm j, \ \calO^\times=\{\pm 1\}, \ R:=\calO \cap \Q(i)=\begin{cases}\Z[\frac{1+i}{2}], \ &\text{if}\ \calO=\calO(q);\\ \Z[i], \ &\text{if}\ \calO=\calO'(q).
 \end{cases} \]


\begin{lemma}\label{lemma:loops}
 Suppose $\ell\nmid 2pq$. Then in the isogeny graph $\calG_{\ell}(\Fpbar)$,
\begin{itemize}
\item[(i)] if  $\frac{1+\pi}{2}\notin \calO$ and $p >q\ell$, then there are $1+{\delta}_{-q}$ loops over the vertex $[E]$;
\item[(ii)] if  $\frac{1+\pi}{2}\in \calO$ and $p >4q\ell$, then there are $1+{\delta}_{-4q}$ loops over the vertex $[E]$.
\end{itemize}
\end{lemma}
\begin{proof} By Deuring's correspondence theorem, a loop in $\calG_{\ell}(\Fpbar)$ corresponds to a principal left ideal $\calO\alpha$ of reduced norm $\ell$. If  $\frac{1+\pi}{2}\notin \calO$, then $\calO= \calO(q)$. For $\alpha=x+\frac{1+i}{2}y+\frac{j-k}{2}z+\frac{ri-k}{q}w\in \calO$, suppose
 \[\Nrd(\alpha)=\left(x+\frac{y}{2}\right)^2+\left(\frac{y}{2}+
 \frac{rw}{q}\right)^2q+ \left(\frac{z}{2}\right)^2p+\left(\frac{z}{2} +\frac{w}{q}\right)^2pq=\ell.\]
If $(z,w)\neq (0,0)$, then $(\frac{z}{2})^2+(\frac{z}{2}+\frac{w}{q})^2q\geq \frac{1}{q}$, and if $p > q\ell$, then $p((\frac{z}{2})^2+(\frac{z}{2}+\frac{w}{q})^2q)>\ell$, impossible. Hence $z=w=0$. Now we need to solve the equation
\begin{equation}\label{eq:2}
\left(x+\frac{y}{2}\right)^2+\frac{y^2q}{4}=\ell
\end{equation}
in $\Z$. This is equivalent to the decomposition of the ideal $\ell R$ in the ring $R=\calO \cap \Q(i)=\Z[\frac{1+i}{2}]$. Since the discriminant of $R$ is $-q$, by Theorem~\ref{theo:K}, \eqref{eq:2} is solvable over $\Z$ if and only if $\big(\frac{-q}{\ell}\big)=1$ and $H_{-q}(x)$ splits into linear factors in $\F_{\ell}[x]$. When this is the case, \eqref{eq:2} has two pairs of solutions up to units in $R^\times=\calO^\times=\{\pm 1\}$, corresponding to two different principal left ideals of $\calO$ of reduced norm $\ell$. Hence there are two loops over $[E]$.

If  $\frac{1+\pi}{2}\in \calO$, then  $\calO=\calO'(q)$. Suppose $\alpha=x+\frac{1+j}{2}y+iz+\frac{r'i-k}{2q}\in \calO$ such that
\[\Nrd(\alpha)=\left(x+\frac{y}{2}\right)^2+
\frac{y^2 p}{4})+\left(z+\frac{r'w}{2q}\right)^2q+\frac{pw^2}{4q} =\ell.\]
If $(y,w)\neq (0,0)$, then $\frac{y^2}{4}+\frac{w^2}{4q} \geq \frac{1}{4q}$, and if $p > 4q\ell$, then $p((\frac{y}{2})^2+(\frac{w}{2q})^2q)>\ell$, impossible. Hence $y=w=0$. Now we need to solve the equation
\begin{equation}\label{eq:4}
x^2+z^2q=\ell
\end{equation}
in $\Z$. This is equivalent to the decomposition of $\ell R$ in $R=\calO \cap \Q(i)=\Z[i]$. In this case $R$ is of discriminant $-4q$, then by Theorem~\ref{theo:K}, \eqref{eq:4} is solvable over $\Z$ if and only of $\big(\frac{-4q}{\ell}\big)=1$ and $H_{-4q}(x)$ splits into linear factors in $\F_{\ell}[x]$. When this is the case, \eqref{eq:4}) has two pairs of solutions up to units in $R^\times=\calO^\times$. Thus $\calO$ has two principal left ideals of reduced norm $\ell$, corresponding to two loops over $[E]$.
\end{proof}

\begin{remark}\label{remk:1} When the assumption of the above Lemma is satisfied, by the proof above, if  $[E]$ has two loops, then the corresponding  $\alpha \notin \End(E) \cap \Q(\pi)= \calO \cap \Q(j)$. This means the loops are not defined over $\F_p$, since $ \End_{\F_p}(E) \subseteq  \End(E)\cap \Q(\pi)$.
\end{remark}

\begin{proof}[Proof of Theorem~\ref{theo:ww}] If
every edge (except loops) in $\calG_{\ell}(\Fpbar)$ has multiplicity one, then the number of vertices adjacent to $[E]$ as predicted by the Theorem is correct.
Let $X_\ell$ be the set of all  left $\calO$-ideals of reduced norm $\ell$. The first part of the Theorem is reduced  to show that any non-principal left $\calO$-ideal $J\in X_\ell$ of reduced norm $\ell$ is not equivalent to  other ideals  in $X_\ell$. We prove this by contradiction.

Assume that there exists some $I \in X_{\ell}-\{J\}$ and $\mu \in B_{p,\infty}^\times$ such that $J=I\mu$, then $\Nrd(\mu)=1$ and $\ell\mu \in J$.

If $\calO= \calO(q)$, write $\ell\mu=x+\frac{1+i}{2}y+\frac{j-k}{2}z+\frac{ri-k}{q}w$ in $\calO$. Then
\[\Nrd(\ell\mu)=\ell^2=\left(x+\frac{y}{2}\right)^2+\left(\frac{y}{2}+
\frac{rw}{q}\right)^2q+\frac{z^2 p}{2}+ \left(\frac{z}{2}+\frac{w}{q}\right)^2pq=\ell^2.\]
If $(z,w)\neq (0,0)$, then $(\frac{z}{2})^2+(\frac{z}{2}+\frac{w}{q})^2q\geq \frac{1}{q}$, and if $p > q{\ell}^2$, then $p((\frac{z}{2})^2+(\frac{z}{2}+\frac{w}{q})^2q)>{\ell}^2$, impossible. Hence $z=w=0$. Now we need to solve the equation
 \begin{equation}\label{eq:3}
 \left(x+\frac{y}{2}\right)^2+ \frac{q y^2}{4}=\ell^2
\end{equation}
in $\Z$.  Note that $(x,y)=(\pm\ell, 0)$ are  trivial solutions of \eqref{eq:3}. In these cases $\mu=\pm1$ and $J=I$ which is a contradiction. If there is a nontrivial solution of \eqref{eq:3}, then $\ell \mu R={\mathfrak{l}}^2$ or $\ell \mu R={\bar{\mathfrak{l}}}^2$. Since $q\equiv 3 \bmod 4$, the class number of $R$ is odd,  $\mathfrak{l}$ and $\bar{\mathfrak{l}}$ are both principal prime ideals of norm $\ell$ of $R$. This implies that $\delta_{-q}=1$ and $\ell R=  \mathfrak{l}\cdot \bar{\mathfrak{l}}$ splits in $R$.
Since $\ell R+(\ell\mu)R \subseteq J$, we have either $\mathfrak{l} \subseteq J$ or $\bar{\mathfrak{l}} \subseteq J$ and  hence $J=\calO \mathfrak{l}$ or $\calO \bar{\mathfrak{l}}$  is a principal left ideal in $X_{\ell}$. This is also a contradiction.
The case for $\calO=\calO'(q)$ can be proved similarly and we omit the proof here.

For the last statement, consider the $\ell$-isogenies starting from $E$, as pointed out in~\cite[Theorem 2.7]{GD16}, there are exactly $1+\big(\frac{-p}{\ell}\big)$ isogenies defined over $\F_p$, as the loops are not defined over $\F_p$ and the multiplicity of each edge (not including the loops) is one, there are at least $1+\big(\frac{-p}{\ell}\big)$ neighbors of $[E]$ in $\calG_{\ell}(\Fpbar)$ defined over $\F_p$. In the following, we will prove that when $p>q{\ell}^2$ in the first case or $p>4q{\ell}^2$ in the second case, there are exactly $1+\big(\frac{-p}{\ell}\big)$ neighbors of $[E]$ defined over $\F_p$.

Again we only show the case $\calO=\calO(q)$. The other case follows by the same argument. For an ideal $I \in X_{\ell}$, let $E_I$ denote the elliptic curve connecting with $E$ by the isogeny $\phi_I$. Then $E_I$ is defined over $\F_p$ if and only if $\calO_R(I)\cong\End(E_I)$ contains an element $\mu$ such that $\mu^2=-p$ according to Theorem~\ref{theo:V} and Lemma~\ref{lemma:unique}. Since $\ell\calO \subseteq \calO_R(I) \subseteq \frac{1}{\ell}\calO$, we may assume  $\mu=\frac{1}{\ell}(a+b\frac{1+i}{2}+c\frac{j-k}{2}+d\frac{ri-k}{q}) \in \frac{1}{\ell}\calO$. By $\mu^2=-p$, then $b=-2a$ and
 \[\left(-a+\frac{dr}{q}\right)^2q+\left(\frac{c}{2}\right)^2p+\left(
 \frac{c}{2} +\frac{d}{q}\right)^2pq=p\ell^2. \]
Thus $p \mid (-aq+dr)$. If $-aq+br\neq 0$, when $p >q\ell^2$, then $\frac{(-qa+dr)^2}{q}>p\ell^2$, not possible. Hence $-a+\frac{dr}{q}=0$ and $q \mid d$. Then $\mu=\frac{1}{\ell}(\frac{c}{2}-(\frac{c}{2}+\frac{d}{q})i)j$ with $(c,d)$ satisfying the equation
\begin{equation}\label{eq:6}
\frac{c^2}{4}+\left(\frac{c}{2}+\frac{d}{q}\right)^2q=\ell^2.
\end{equation}
Each solution of \eqref{eq:6} corresponds to a principal ideal in $R=\Z[\frac{1+i}{2}]$ of norm $\ell^2$. Since the class number of $R$ is odd when $q \equiv 3 \bmod 4$,  $\frac{c}{2}-(\frac{c}{2}+\frac{d}{q})i$ is either $\pm \ell$ or $\pm{\alpha}^2, \pm{\bar{\alpha}}^2$ if $R$ has a principal ideal $R\alpha$ of norm $\ell$. Thus either $\mu=\pm j$, or when $[E]$ has loops,  $\mu=\pm\frac{1}{\ell}\alpha^2j$ or $\pm\frac{1}{\ell}\bar{\alpha}^2j$.

We now follow the notations and ideas in the proof of \cite[Theorem 5]{YZ}. There is a ring isomorphism $\theta:\ \calO/{\ell\calO}\rightarrow M_2(\F_\ell)$ by
\[1\mapsto \begin{pmatrix}                
    1 & 0\\
    0 & 1
\end{pmatrix},\ i\mapsto \begin{pmatrix}                
    0 & -q\\
    1 & 0
\end{pmatrix},\ j\mapsto \begin{pmatrix}                
    u & qv\\
    v & -u
\end{pmatrix}\]
where $(u,\ v)$ is a solution of $u^2+qv^2= -p$ in $\F_\ell$. Let $\iota:\calO\rightarrow \calO/\ell\calO$ be the restriction map.
The set $\overline{X}_{\ell}$ of the $\ell+1$ left ideals of $M_2(\F_{\ell})$ is
\[\overline{X}_{\ell}:=\{M_2(\F_\ell)\omega, M_2(\F_\ell)\omega_a\mid a\in \F_\ell\} \]
where $\omega:=\begin{pmatrix}0 & 0\\0 & 1\end{pmatrix}$ and $\omega_a:=\begin{pmatrix}1 & a\\0 & 0\end{pmatrix}$. Under the map $\theta\circ\iota$, there is a $1$-to $1$ correspondence of $X_\ell$ and $\overline{X}_{\ell}$ compatible with multiplication. Thus
we only need to check: (i) for which ideal $I \in \overline{X}_{\ell}$,  $I\theta(j)\subseteq I$; (ii)  when $E$ has loops,  for which ideal $I \in \bar{X_{\ell}}$,  $I\theta(\alpha^2j) \subseteq \ell I= \{0\}$ or $I\theta(\bar{\alpha}^2j) = \{0\}$. Since $\det(\theta(j))=p \neq 0$ in $\F_{\ell}$, to check (ii), it suffices to check: (iii) for which ideal $I \in \bar{X_{\ell}}$,  $I\theta(\alpha^2) \subseteq \ell I= \{0\}$ or $I\theta(\bar{\alpha}^2) = \{0\}$.

When $\big(\frac{-p}{\ell}\big)=1$, we take $(u,v)=(u,0)$ where $u^2= -p\in \F_\ell$. Then $\theta(j)=\begin{pmatrix}u & 0\\0& -u\end{pmatrix}$. By computation,
 \[ \omega \theta(j)\in M_2(\F_{\ell})\omega,\ \omega_0 \theta(j)\in M_2(\F_{\ell})\omega_0,\ \omega_a \theta(j)\notin M_2(\F_{\ell})\omega_a\ (a\neq 0). \]
Hence there are exactly two ideals $I_1=\calO\ell+\calO(u+j)$ and $I_2=\calO\ell+\calO(u-j)$ in $X_{\ell}$ such that $\pm j \in \calO_R(I_1)$ and $\calO_R(I_2)$. They correspond to two edges starting from $[E]$ in $\calG_{\ell}(\Fpbar)$.
If $\big(\frac{-p}{\ell}\big)=-1$, by computation there is no $I\in \bar{X_{\ell}}$ such that $I\theta(j)\subseteq I$.

When $[E]$ has two loops, let $\alpha=x+\frac{1+i}{2}y\in R$ such that $\alpha\bar{\alpha}=\ell$, then $\ell \nmid y$, and
 \[ \theta(\alpha^2)=\begin{pmatrix}2(x+\frac{y}{2})^2 & -(x+\frac{y}{2})yq\\(x+\frac{y}{2})y& 2(x+\frac{y}{2})^2\end{pmatrix},\ \ \theta(\bar{\alpha}^2)=\begin{pmatrix}2(x+\frac{y}{2})^2 & (x+\frac{y}{2})yq\\-(x+\frac{y}{2})y& 2(x+\frac{y}{2})^2\end{pmatrix}. \]
Let $b=2\frac{x}{y}+1$ in $\F_{\ell}$. Then only
\[\omega_{- b}\theta(\alpha^2)=0,\ \omega_{b}\theta(\bar{\alpha}^2)=0.\]
These two ideals correspond to the principal left ideals $\calO{\alpha}$ and $\calO\bar{\alpha}$ in $X_{\ell}$. Thus, except the loops, there are at most two $E_I$ defined over $\F_p$. Since the multiplicity of each edges is one, there are $1+\big(\frac{-p}{\ell}\big)$ neighbors of $[E]$ in $\calG_{\ell}(\Fpbar)$ defined over $\F_p$.
\end{proof}

\begin{remark} Fix $\ell$ and $q$, then  the bound $q{\ell}^2$ or $4q{\ell}^2$ for $p$ is sharp, just like the case considered in \cite{YZ}. We have two examples. In both cases, the result in Theorem~\ref{theo:ww} does not hold when the bound is not satisfied.

(1) Let $q=11$, $\ell=13$. Then $p=1847$ is the largest prime such that $\big(\frac{-p}{q}\big)=1$ and $p<q{\ell}^2$. Let $E: y^2=x^3+1594x+447$, then $E$ is a supersingular elliptic curve defined over $\F_{1847}$ with $\End(E)\cong \calO(11)$. By computation, $[E]$ has three neighbors $j_1=1336$, $j_2=319$ and $j_3=437$ defined over $\F_{1847}$ in $\calG_{13}(\overline{\F}_{1847})$, which is larger than $1+\big(\frac{-p}{\ell}\big)=2$.

(2) Let $q=3$, $\ell=5$. Then $p=293$ is the largest prime such that $\big(\frac{-p}{q}\big)=1$ and $p<4q{\ell}^2$. Let $E: y^2=x^3+256x+73$, then $E$ is supersingular over $\F_{293}$ with $\End(E)\cong \calO'(3)$. $[E]$ has no loops but one neighbor $j_1=212$ defined over $\F_{293}$ in $\calG_{5}(\overline{\F_{293}})$ which is larger than $1+\big(\frac{-p}{\ell}\big)=0$.
\end{remark}

\begin{example}
Let $p=311$, $q=3$, $\ell=5$. The elliptic curve $E: y^2=x^3+122x+185$ is supersingular with $\End(E)\cong \calO'(3)$. In the $5$-isogeny graph $\calG_{5}(\overline{\F}_{311})$, $[E]$ has no loops (as $\big(\frac{-3}{5}\big)=-1$), and only two neighborhoods $j(E_1)=225$, $j(E_2)=19$ defined over $\F_{311}$(as $\big(\frac{-311}{5}\big)=1$). Moreover $\End(E_1)\cong \calO'(67)$ and $\End(E_2)\cong \calO'(419)$.
\end{example}

\section{The bound of $q_j$ for any supersingular $j-$invariant $j$ in $\F_p$}

In this section we identify $K=\Q(\sqrt{-p})=\Q(j)$ with class number  $h=h_K$. Note that
 \[ \begin{split} & O_K=\Z[\sqrt{-p}],\ \ \qquad D_K=-4p\quad (\text{if}\ p\equiv 1\bmod{4}),\\
 & O_K=\Z[\frac{1+\sqrt{-p}}{2}],\quad D_K=-p\quad (\text{if}\ p\equiv 3\bmod{4}). \end{split}  \]
Let $O$ be the order of $K$ of conductor $2$. Then
 \[ O=\Z+2 O_K=\begin{cases} \Z[2\sqrt{-p}], & \text{if}\ p \equiv 1 \pmod 4,\\ \Z[\sqrt{-p}], & \text{if}\ p \equiv 3 \pmod 4.\end{cases} \]
Let $L_0$ and $L_1$ be the Hilbert class field and the ring class field of $O$ over $K$. Then
 \[ \Gal(L_1/K)\cong\cl(O)\cong I_K(2)/P_{K,\Z}(2),\]
 \[ \Gal(L_0/K)\cong \cl(O_K)\cong I_K/P_K\cong I_K(2)/P_K(2).\]
By the inclusion $P_{K,\Z}(2)\subseteq P_K(2)$,  $L_1 \supseteq L_0$. Moreover, from ~\cite[Theorem 7.24]{Cox89},
 \[[L_1:L_0]=h(O)/h=\begin{cases}2, & \text{if}\ p\equiv 1 \pmod 4,\\
 3, & \text{if}\ p\equiv 3 \pmod 8,\\
 1, & \text{if}\ p\equiv 7 \pmod 8.\end{cases}\]
By properties of class fields, we know that $L_0/\Q$ and $L_1/\Q$ are Galois. Let $\zeta_8$ be a primitive eighth root of unity, then $\Q(\zeta_8)$ is a Galois extension of $\Q$. Hence $L_0(\zeta_8)$ and $L_1(\zeta_8)$ are also Galois over $\Q$. By \cite[Lemma 2.11]{Ibukiyama}, we know that if $p\equiv 3\bmod{4}$, $L_0$ and $\Q(\zeta_8)$ are linearly disjoint over $\Q$; if  $p\equiv 1\bmod{4}$, $L_1\cap \Q(\zeta_8)=L_0\cap \Q(\zeta_8)=\Q(\sqrt{-1})$. We have the following figures about field extensions.
\begin{figure}
\begin{minipage}{0.49\textwidth}
\
 \[ \xymatrix{
 & L_1(\zeta_8) \ar@{-}[dl]_{2} \ar@{-}[d]^{2} &  \\
 L_1 \ar@{-}[d]_{2} & L_0(\zeta_8) \ar@{-}[dl]_{2} \ar@{-}[dr]^{h} &  \\
 L_0\ar@{-}[d]_{h} \ar@{-}[dr]^{h} &  & \Q(\zeta_8)\ar@{-}[dl]^{2} \\
   K \ar@{-}[dr]_{2} & \Q(\sqrt{-1}) \ar@{-}[d]_{2} &   \\
    &  \Q &    } \]

\caption{Field extensions when $p\equiv 1\bmod{4}$}
\label{fig:21}
\end{minipage}
\begin{minipage}{0.49\textwidth}

 \[ \xymatrix{
 & L_1(\zeta_8) \ar@{-}[dl]_{4} \ar@{-}[d]^{1\ or \ 3} & \\
 L_1 \ar@{-}[d]_{1\ or \ 3} & L_0(\zeta_8) \ar@{-}[dl]_{4} \ar@{-}[dr]^{2h} &  \\
 L_0\ar@{-}[d]_{h}  &  & \Q(\zeta_8)\ar@{-}[ddl]^{4} \\
   K \ar@{-}[dr]_{2} &  &   \\
    &  \Q &    } \]

\caption{Field extensions when $p\equiv 3\bmod{4}$}
\label{fig:22}
\end{minipage}
\end{figure}

\begin{lemma}~\label{lem:discL_0} For $i=0$ and $1$, let $n_{i}=[L_i(\zeta_8):\Q]$ and $d_{i}=|D_{L_i(\zeta_8)}|$.
 \begin{itemize}
  \item[(i)] $K(\zeta_8)/\Q$ is an abelian extension of degree $8$ and discriminant $2^{16} p^4$.
  \item[(ii)] If $p\equiv 3\bmod{4}$, then $n_{0}=8h$, $d_{0}= 2^{16 h} p^{4 h}$. If furthermore $p\equiv 3\bmod{8}$, then $n_{1}=24 h$, $d_{1}= 2^{52 h}p^{12 h}$.

  \item[(iii)] If $p\equiv 1\bmod{4}$, then $n_{0}=4h$, $n_{1}=8 h$, $d_{0}= 2^{8h}p^{2h}$ and $d_{1}\mid 2^{21 h}p^{4h}$.
 \end{itemize}
\end{lemma}
\begin{proof} For (i), one just needs to compute the discriminant $D_{K(\zeta_8)}$. It is well known $D_{\Q(\zeta_8)}=2^8$. The extension $K(\zeta_8)/\Q(\zeta_8)$ is unramified outside $p$ and tamely ramified over all primes above $p$, hence the different $\fD_{K(\zeta_8)/\Q(\zeta_8)}=\prod_{\fp\mid p\ \text{in}\ \Q(\zeta_8)}\fp^2$ by \cite[Theorem 2.6]{Neuk99}. By Proposition~\ref{prop:disc}(i), we obtain $D_{K(\zeta_8)}$.

The degrees $n_{0}$ and $n_{1}$ follow from the two figures.

Since $L_0$ is the Hilbert class field of $K$ which is the maximal unramified abelian extension of $K$,
$D_{L_0/K}=1$ and by Proposition~\ref{prop:disc}(i),
 \[D_{L_0}=(D_K)^{h}.\]
Note that  $D_{\Q(\zeta_8)}=2^8$, $D_{\Q(\sqrt{-1})}=-2^2$. If $p\equiv 3\bmod{4}$, then $L_0$ and $\Q(\zeta_8)$ are linearly disjoint over $\Q$, and $D_{L_0}$ and $D_{\Q(\zeta_8)}$ are coprime,
by Proposition~\ref{prop:disc}(ii), then
 \[ d_0=D_{L_0(\zeta_8)}= D_{L_0}^{[\Q(\zeta_8):\Q]}\cdot D_{\Q(\zeta_8)}^{[L_0:\Q]}={2}^{16 h}\cdot p^{4h}. \]
If $p\equiv 1\bmod{4}$, $L_0$ and $\Q(\zeta_8)$ are linearly disjoint over $\Q(\sqrt{-1})$. By computation,
\[N_{\Q(\sqrt{-1})/\Q}(D_{L_0/\Q(\sqrt{-1})})=p^{h},\ N_{\Q(\sqrt{-1})/\Q}(D_{\Q(\zeta_8)/\Q(\sqrt{-1})})=2^{4}.\]
 Thus $D_{L_0/\Q(\sqrt{-1})}$ and $D_{\Q(\zeta_8)/\Q(\sqrt{-1})}$ are coprime. By Proposition~\ref{prop:disc}(ii), then
 \[ d_{0}= (D_{\Q(\zeta_8)})^{h} (D_{L_0})^2 (D_{\Q(\sqrt{-1})})^{-2h} =2^{8 h} p^{2 h}. \]

To compute $d_{1}$, note that $L_1/K$ is ramified only at primes above $2$ and $L_0/K$ is  unramified, then $L_1/L_0$ and $L_1(\zeta_8)/L_0$ are ramified only at primes above $2$. Note that $L_1/\Q$ is Galois, $L_1/L_0$ is of degree $2$ or $3$, all primes of $L_0$ above $2$ must be totally ramified in $L_1$. We also know $2$ is totally ramified in $\Q(\zeta_8)/\Q$. Let $e$, $f$ and $g$ be the ramification index, the degree of the residue extension and the number of primes above $2$ in $L_1(\zeta_8)$. Then $efg=n_{L_1}$ and $2O_{L_1(\zeta_8)}$ has the prime decomposition
 \[ 2O_{L_1(\zeta_8)}=\prod_{i=1}^g \fP^e_{1,i}. \]

 If $p\equiv 3\bmod{8}$, then primes above $2$ are  unramified in $L_0/\Q$. We  find that $e=12$, $fg=2h$ and all primes above $2$ in $L_0(\zeta_8)$ are totally (tamely) ramified in $L_1(\zeta_8)$. By \cite[Theorem 2.6]{Neuk99}, the different of $L_1(\zeta_8)/L_0(\zeta_8)$ is
\[\mathfrak{D}_{L_1(\zeta_8)/L_0(\zeta_8)}=\prod_{i=1}^g \mathfrak{P}_{1,i}^2.\]
Hence
 \[ d_{1}=d_{0}^3 N_{L_1(\zeta_8)/\Q}(\mathfrak{D}_{L_1(\zeta_8)/L_0(\zeta_8)}) = 2^{52 h}p^{12h}. \]

If $p \equiv 1 \bmod 4$, then either $e=4$ or $8$. If $e=4$, primes above $2$ are unramified in $L_1(\zeta_8)/L_0(\zeta_8)$ and the different of $L_1(\zeta_8)/L_0(\zeta_8)$ is $(1)$. In this case $d_{1}=d_{0}^2$. If $e=8$, then $fg=h$ and  primes above $2$ are wildly ramified in $L_1(\zeta_8)/L_0(\zeta_8)$. By \cite[Theorem 2.6]{Neuk99}, the different of $L_1(\zeta_8)/L_0(\zeta_8)$ is
\[\mathfrak{D}_{L_1(\zeta_8)/L_0(\zeta_8)}=\prod_{i=1}^g \mathfrak{P}_{1,i}^m,\ \text{where}\  1 \leq m\leq 5.\]
Hence
 \[  d_{1}=(d_{0})^2 N_{L_1(\zeta_8)/\Q}(\mathfrak{D}_{L_1 (\zeta_8)/L_0(\zeta_8)})\mid  2^{21 h}p^{4 h}. \qedhere \]
 \end{proof}

\begin{lemma} \label{lemma:frob} Let $q$ and $q'$ be distinct primes. Let $\sigma_3=(\zeta_8\mapsto \zeta_8^3)\in \Gal(\Q(\zeta_8)/\Q)$. Then
\begin{itemize}
 \item[(i)] $q\equiv 3\bmod{8}$ and $\left(\frac{q}{p}\right)=-1$ if and only if $[\frac{K(\zeta_8)/\Q}{q}]\in \Gal(K(\zeta_8)/\Q)$ is the unique element $\Delta$ such that $\Delta|_K=\id$ and $\Delta|_{\Q(\zeta_8)}=\sigma_3$.
 \item[(ii)] The conditions that $q$ and $q'$ satisfy \eqref{eq:1} and  $\calO(q) \cong \calO(q')$ (resp.  $\calO'(q) \cong \calO'(q')$) is equivalent to that $[\frac{\Q(\zeta_8)/\Q}{q}]=[\frac{\Q(\zeta_8)/\Q}{q'}]=\sigma_3$ and $[\frac{L_1/\Q}{q}]=[\frac{L_1/\Q}{q'}]$ (resp. $[\frac{L_0/\Q}{q}]=[\frac{L_0/\Q}{q'}]$).
\end{itemize}
\end{lemma}
\begin{proof} The condition that $q\equiv 3\bmod{8}$ is equivalent to $[\frac{\Q(\zeta_8/\Q}{q}]=\sigma_3\in \Gal(\Q(\zeta_8)/\Q)$. The condition $(\frac{p}{q})=-1$ is equivalent to that $q$  splits in $K$, i.e., $[\frac{K/\Q}{q}]=[\frac{K/\Q}{q'}]=1$. So (i) holds.

Let $q=\fq\bar{\fq}$ and $q=\fq'\bar{\fq'}$ be the factorization of $q$ and $q'$ in $K$.
By Lemma ~\ref{lem:isomophic}(ii), the condition that $\calO(q) \cong \calO(q')$
is equivalent to
 \[\{\big[\frac{L_1/K}{{\mathfrak{q}}}\big], \big[\frac{L_1/K}{{\mathfrak{q}}}\big]^{-1} \}=\{\big[\frac{L_1/K}{{\mathfrak{q}}'}\big], \big[\frac{L_1/K}{{\mathfrak{q}}'}\big]^{-1}\}. \]
Let $\tau$ be a lifting of $(\sqrt{-p}\mapsto -\sqrt{-p})\in \Gal(K/\Q)$ in $\Gal(L_1/\Q)$ and let $\fQ$ be a prime of $L_1$ above $q$, then
$[\frac{L_1/\Q}{q}]=\{[\frac{L_1/\Q}{\fQ}], \tau [\frac{L_1/\Q}{\fQ}]\tau^{-1}\}$ (these two probably equal). When $[\frac{K/\Q}{q}]=1$, this set is equal to $\{[\frac{L_1/K}{{\mathfrak{q}}}], [\frac{L_1/K}{{\mathfrak{q}}}]^{-1}\}$.
Hence  $\calO(q) \cong \calO(q')$ and $(\frac{p}{q})=(\frac{p}{q'})=-1$ is equivalent to
 \[\big[\frac{L_1/\Q}{q}\big]=\big[\frac{L_1/\Q}{q'}\big].\]
The case for $\calO'$ follows similarly.
\end{proof}
\begin{lemma} \label{lemma:bound} Let $\gamma$ be any element in $\Gal(K(\zeta_8)/\Q)$, $C_0$ and $C_1$ be any conjugacy class in $\Gal(L_0(\zeta_8)/\Q)$ and $\Gal(L_1(\zeta_8)/\Q)$ respectively. Assuming GRH.
 \begin{itemize}
 \item[(i)] For constant $c>0$, $\pi_{\gamma}(c\sqrt{p}, K(\zeta_8)/\Q)\sim \frac{c\sqrt{p}}{4\log p}$ as $p\rightarrow \infty$.
 \item[(ii)] Suppose $p>2000$ and $x\geq p\log^4 p$, then
  \[ \frac{d_0}{|C_0|}\pi_{C_0}(x, L_0(\zeta_8)/\Q)\frac{\log x}{\sqrt{x}} \geq \begin{cases} \sqrt{x}-1.28 h\log^2 x,\ &\text{if}\ p\equiv 1\bmod 4;\\
  \sqrt{x}-2.56 h\log^2 x,\ &\text{if}\ p\equiv 3\bmod 4.
  \end{cases} \]
  \[ \frac{d_1}{|C_1|}\pi_{C_1}(x, L_1(\zeta_8)/\Q)\frac{\log x}{\sqrt{x}}  \geq \begin{cases}\sqrt{x}-2.67 h\log^2 x,\ &\text{if}\ p\equiv 1\bmod 4;\\
  \sqrt{x}-7.76 h\log^2 x,\ &\text{if}\ p\equiv 3\bmod 8.
  \end{cases} \]

 Suppose $p>2000$ and $x\geq p\log^6 p$, then
  \[ \frac{d_0}{|C_0|}\pi_{C_0}(x, L_0(\zeta_8)/\Q)\frac{\log x}{\sqrt{x}} \geq \begin{cases} \sqrt{x}-1.12 h\log^2 x,\ &\text{if}\ p\equiv 1\bmod 4;\\
  \sqrt{x}-2.24 h\log^2 x,\ &\text{if}\ p\equiv 3\bmod 4.
  \end{cases} \]
  \[ \frac{d_1}{|C_1|}\pi_{C_1}(x, L_1(\zeta_8)/\Q)\frac{\log x}{\sqrt{x}}  \geq \begin{cases}\sqrt{x}-2.33 h\log^2 x,\ &\text{if}\ p\equiv 1\bmod 4;\\
  \sqrt{x}-6.80 h\log^2 x,\ &\text{if}\ p\equiv 3\bmod 8.
  \end{cases} \]
 \end{itemize}
\end{lemma}
\begin{proof} We shall use the explicit formula \eqref{eq:cheb} in the Chebotarev density Theorem (Theorem~\ref{them:Cheb}).

For (i), consider the extension $K(\zeta_8)/\Q$, then $d_{K(\zeta_8)}=2^{16} p^4$ and $n_{K(\zeta_8)}=8$. Take $x=c\sqrt{p}$, the main term in \eqref{eq:cheb} is $2c\sqrt{p}/\log p$, the error term is of order $p^{\frac{1}{4}} \log p$. When $p\rightarrow \infty$, we get (i).

For (ii),  consider the case $L_1/\Q$ and $p\equiv 3\bmod 8$ case. The other cases can be treated similarly. In this case $n_1=24h$ and $d_1=2^{52h} p^{12h}$. Note that if $x>2000$,
 \[ \int_2^x\frac{dt}{\log t} =\frac{x}{\log x}-\frac{2}{\log 2}+ \int_2^x\frac{dt}{\log^2 t} \geq \frac{x}{\log x}. \]
By \eqref{eq:cheb}, if $x>2000$, then
 \[ \begin{split} & \frac{d_1}{|C_1|}\pi_{C_1}(x, L_1(\zeta_8)/\Q)\frac{\log x}{\sqrt{x}}\\   \geq &
 \sqrt{x}-h\log^2 x\left [ \frac{36\log p}{\log^2 x}+\frac{156\log 2+ 144}{\log^2 x} + \frac{52\log 2+6}{\pi\log x}+\frac{12}{\pi}\frac{\log p}{\log x} +\frac{3}{\pi}\right]. \end{split} \]
Note that $\frac{\log p}{\log x}\leq 1$ if $x\geq p$.  When $p$ is fixed and $x\geq p\log^4 p$ increases, the other terms inside $[\ \  ]$  of the above inequality decrease; when $p$ increases and $x= p\log^4 p$ or $p\log^6 p$,  the other terms inside $[\ \ ]$   also decrease. This leads to the bound in (ii).
\end{proof}

\begin{proof}[Proof of Theorem~\ref{theo:bound}]
(1) By the Brauer-Siegel Theorem, the number of supersingular $j$ over $\F_p$ is of order $O(h)=O(\sqrt{p})$, but by Lemma~\ref{lemma:bound}(i), there are only $O(\frac{\sqrt{p}}{\log p})$ many $q<C\sqrt{p}$ satisfying $q\equiv 3\bmod{8}$ and $(\frac{p}{q})=-1$ when $p\rightarrow \infty$, hence (1) holds.

(2) For $p<2000$, we check numerically in the appendix that $q_j<p\log^2 p$. Suppose $p>2000$. It suffices to find $x$ such that $\pi_{C_i}(x, L_i(\zeta_8)/\Q)>0$ for any conjugacy class $C_i$. By Lemma~\ref{lemma:bound}(ii), to have $\pi_{C_i}(x, L_i(\zeta_8)/\Q)>0$, it suffices to find $x\geq p\log^4 p$, such that $\sqrt{x}-C h \log^2 x>0$  for different $C$ there.
By the Brauer-Siegel Theorem, when $p$ is sufficiently large, $h\sim \sqrt{p}$ if $p\equiv 3\bmod 4$ or $2\sqrt{p}$  if $p\equiv 1\bmod 4$. Replace $h$ by $\sqrt{p}$ or $2\sqrt{p}$, we just need to find $x\geq p\log^4 p$, such that $\sqrt{x}-7.76 \sqrt{p} \log^2 x>0$. This is satisfied if $p>2000$ and $x=10000 p\log^4 p$.

(3) Suppose $p>2000$. It suffices to find $x$ such that $\pi_{C_i}(x, L_i(\zeta_8)/\Q)>0$ for any conjugacy class $C_i$. By Lemma~\ref{lemma:bound}(ii), we just need to find $x\geq p\log^6 p$ such that $\sqrt{x}-C h\log^2 x>0$ for different $C$ there. By \cite[Exercise 5.27]{Coh96},  $h<\sqrt{p}\log p$ if $p\equiv 3\bmod{4}$ and  $h<\sqrt{4p}\log (4p)$ if $p\equiv 1\bmod{4}$. We thus only need to find $x\geq p\log^6 p$ such that $\sqrt{x}-6.8 \sqrt{p}\log{p} \log^2 x>0$. Take $x= 10000 p \log^6p$, we can check  $\sqrt{x}-6.8 \sqrt{p}\log{p} \log^2 x>0$.

(4) Let $q_1$, $q_2$ be two distinct primes satisfying \eqref{eq:1}. If $(x,y)$ is an integer solution of  $x^2+4py^2=q_1q_2$, $y$ must be even since $q_1q_2 \equiv 1 \bmod 8$ and $x^2 \equiv 0,1,4 \bmod 8$. Thus  $x^2+4py^2=q_1q_2$ has integer solutions is equivalent to  $x^2+16py^2=q_1q_2$ has integer solutions. Thus if both $q_1$ and $q_2 <4\sqrt{p}$, the equation has no integer solution and
$\calO(q_1) \not\cong \calO(q_2)$ by Lemma~\ref{lem:isomophic}(ii). Similarly by Lemma~\ref{lem:isomophic}(iii), if both $q_1$ and $q_2 <\frac{\sqrt{p}}{2}$, the equation $x^2+py^2=4q_1q_2$ has no integer solutions, and $\calO'(q_1)\not\cong\calO'(q_2)$.  Then if $p \equiv1 \bmod 4$,  $N(4\sqrt{p})=\pi_\Delta(4\sqrt{p}, K(\zeta_8)/\Q)$. If $p \equiv 3 \bmod 4$, $N(\frac{1}{2}\sqrt{p})= 2 \pi_\Delta(\frac{1}{2}\sqrt{p}, K(\zeta_8)/\Q)$ and $N(4\sqrt{p})\geq  \pi_\Delta(4\sqrt{p}, K(\zeta_8)/\Q)+\pi_\Delta(\frac{1}{2}\sqrt{p}, K(\zeta_8)/\Q)$.
\end{proof}

\vskip 1cm
\newcommand{\etalchar}[1]{$^{#1}$}

\appendix
\section{Comparing $M(p)$ with $\sqrt{p}$ and ${p\log^2 p}$ when $p<2000$}
For a prime $p>3$, let $\M(p)$ be the maximal value of $q_j$ for $j$ a supersingular invariant over $\F_p$ defined in \S~1. The following two tables list the values of $\M(p)$ for all $p<2000$ and compare it with $\sqrt{p}$ and $p\log^2 p$.

In the following we present our algorithms to compute Table 1 and Table 2. For a finite set A, let $|A|$ denote the cardinality of $A$.
\vspace{3mm}

\noindent {\bf Algorithm~1}

\noindent Input: Prime $p \equiv 1 \bmod 4$.

\noindent Output: The value $M(p)$.

\noindent Procedure:

\begin{enumerate}
\item
Compute the set $\SSj(p)$ of all supersingular $j$-invariants in $\F_p$.

\item
Set $\SE(p)=$ the empty set.

\item  For prime $3 \leq q \leq p\log^2 p$ such that $\big(\frac{-p}{q}\big)=1$ and $q\equiv 3 \bmod 8$, compute the $j$-invariant $j_q\in F_p$ such that $\End(E_{j_q})\cong \calO(q)$. More precisely,
    \begin{itemize}
    \item[(3.1)] let $v(d)$ be the set of roots of Hilbert class polynomial $H_{d}$ in $\F_p$. Compute $v(-q)$, $v(-4p)$ and $v(-(\frac{4(r^2+p)}{q}))$.
    \item[(3.2)] compute $A=v(-q)\cap v(-4p)\cap v(-(\frac{4(r^2+p)}{q}))$, if $|A|=1$, return $A$, otherwise return $A=$ the empty set.
    \end{itemize}

\item  Set $\SE(p)=SE(p) \cup A$. Repeat Step 3 until $|\SE(p)|=|\SSj(p)|$. Return $q$.
\end{enumerate}

\begin{remark}
Recall that when $p \equiv 1 \bmod 4$, for a supersingular elliptic curve $E$ defined over $\F_p$, we have $\End(E)\cong \calO(q)$ for some $q$ satisfying
\begin{equation}\label{condtion:q}
\big(\frac{-p}{q}\big)=1\ \text{and} \ q\equiv 3 \bmod 8
\end{equation}
 Here we do a loop for $q$ satisfying \eqref{condtion:q} in an ascending order, and  compute the corresponding $j$-invariant $j$, then make them into a set. In this way, if in some step, we get the equality $\SE(p)=\SSj(p)$, then we get the maximal $q_j$.

One thing needed to explain is the following: in Step 3, we compute the associated supersingular $j$-invariant $j_q$ of $q$ by computing the common roots of $H_{-q}$, $H_{-4p}$ and $H_{-(\frac{4(r^2+p)}{q})}$ in $\F_p$. Since by ~\cite[Theorem 3]{CG14}, $j_q$ is a root of $H_{-d}$ if and only if $\calO^T(q)=\Z\langle i,j-k,\frac{2(ri-k)}{q}\rangle$ has an element of reduced norm $d$. This is the case since $i, 2j,\frac{ri-k}{2} \in \calO^T(q)$ are of reduced norm $q, 4p$ and $\frac{4(r^2+p)}{q}$ respectively. Thus if $v(-q)\cap v(-4p)\cap v(-(\frac{4(r^2+p)}{q}))$ has just one element, it must be $j_q$. If it has more than one element, we quit this $q$ and do Step 3 for the next $q$. Thus the output of algorithm 1 is equal or larger than the real $\M(p)$. But in our experiment, we find the intersection of these three sets always has one element. Anyway, the data in Table 1 and Table 2 is enough to show that $\M(p)<p\log^2 p$.
\end{remark}

\noindent {\bf Algorithm~2}

\noindent Input: Prime $p \equiv 3 \bmod 4$.

\noindent Output: The value $M(p)$.

\noindent Procedure:

\begin{enumerate}
\item Compute the set  $\SSj(p)$ of all supersingular $j$-invariants with $j \in \F_p\setminus \{1728\}$.

\item Set $\SE(p)$ and $\XE(p)$ to be the empty sets.

\item  For all prime $3 \leq q \leq p\log^2p$ such that $\big(\frac{-p}{q}\big)=1$ and $q\equiv 3 \bmod 8$, do
      \begin{itemize}

      \item [(3.1)] compute the $j$-invariant $j_q\in \F_p$ such that $\End(E_{j_q})\cong \calO(q)$ as in Algorithm 1. If $j_q \neq 1728$, set $\SE(p)=\SE(p) \cup \{j_q\}$, otherwise, set $\SE(p)=\SE(p) \cup \emptyset$
      \item [(3.2)] compute the prime ideal decomposition of $q$ in $K=\Q(\sqrt{-p})$: $(q)=\mathfrak{q}_1\mathfrak{q}_2$. If $\mathfrak{q}_1$ is not principal, set $\XE(p)=\XE(p)\cup \{[\mathfrak{q}_1],[\mathfrak{q}_2]\}$ where $[\mathfrak{q}_1]$ and $[\mathfrak{q}_2]$  are the ideal classes in the class group of $K$. Otherwise, set $\XE(p)=\XE(p) \cup \emptyset$.
      \end{itemize}

\item  Compare $|\SSj(p)|$ and $|\SE(p)|+\frac{|\XE(p)|}{2}$. If they are equal, return $q$. Otherwise repeat Step 3.
\end{enumerate}

\begin{remark}
When $p \equiv 3 \bmod 4$, for a supersingular elliptic curve $E$ defined over $\F_p$, we have $\End(E)\cong \calO(q)$ or $\calO'(q)$ for some $q$ satisfying~\eqref{condtion:q}, and for $j \neq 1728$, $\calO(q)\ncong \calO'(q)$ by Lemma~\ref{lem:isomophic}(i).
Here, we do a loop for $q$ satisfying \eqref{condtion:q} in an ascending order. First we compute the $j-$invariant $j_q$ such that $\End(E_{j_q})\cong \calO(q)$ for each $q$ as in algorithm 1. For the $j$-invariant of $\calO'(q)$, if we compute the other three Hilbert class polynomials as in the case of $\calO(q)$, the running time is very expensive, thus we use another way. We define a set $\XE(p)$ consisting of ideal classes $[\mathfrak{q}_1]$ and $[\mathfrak{q}_2]$ if they are not equal, and each correspond to one supersingular $j$-invariant $j'$ such that $\End(E_{j'})\cong \calO'(q)$, if $[\mathfrak{q}_1]=[\mathfrak{q}_2]=1$, they correspond to $j=1728$ by Lemma~\ref{lem:isomophic}(iii). Thus when $|\SSj(p)| =|\SE(p,r)|+\frac{|\XE(p,r)|}{2}$, we obtain the maximal $q_j$.
\end{remark}

\begin{table}[htbp]
  \centering
  \caption{The data of prime $p \equiv 1 \bmod 4$}
    \centering
    \begin{tabular}{||r|r|r|r||r|r|r|r||r|r|r|r||}
    \hline
    \hline

    $p$ &  $\M(p)$ & $\frac{\M(p)}{\sqrt{p}}$ & $\frac{\M(p)}{p\log^2 p}$ & $p$ &  $\M(p)$ & $\frac{\M(p)}{\sqrt{p}}$ & $\frac{\M(p)}{p\log^2 p}$ & $p$ &  $\M(p)$ & $\frac{\M(p)}{\sqrt{p}}$ & $\frac{\M(p)}{p\log^2 p}$ \\ \hline

    5     & 3     & 1.34  & 0.232  & 557   & 491   & 20.80  & 0.022  & 1193  & 1483  & 42.94  & 0.025  \\
    13    & 11    & 3.05  & 0.129  & 569   & 4219  & 176.87  & 0.184  & 1201  & 283   & 8.17  & 0.005  \\
    17    & 11    & 2.67  & 0.081  & 577   & 331   & 13.78  & 0.014  & 1213  & 619   & 17.77  & 0.010  \\
    29    & 19    & 3.53  & 0.058  & 593   & 587   & 24.11  & 0.024  & 1217  & 1499  & 42.97  & 0.024  \\
    37    & 19    & 3.12  & 0.039  & 601   & 811   & 33.08  & 0.033  & 1229  & 1987  & 56.68  & 0.032  \\
    41    & 211   & 32.95  & 0.373  & 613   & 307   & 12.40  & 0.012  & 1237  & 739   & 21.01  & 0.012  \\
    53    & 67    & 9.20  & 0.080  & 617   & 379   & 15.26  & 0.015  & 1249  & 2003  & 56.68  & 0.032  \\
    61    & 59    & 7.55  & 0.057  & 641   & 1787  & 70.58  & 0.067  & 1277  & 1499  & 41.95  & 0.023  \\
    73    & 43    & 5.03  & 0.032  & 653   & 491   & 19.21  & 0.018  & 1289  & 1091  & 30.39  & 0.017  \\
    89    & 163   & 17.28  & 0.091  & 661   & 571   & 22.21  & 0.020  & 1297  & 179   & 4.97  & 0.003  \\
    97    & 59    & 5.99  & 0.029  & 673   & 107   & 4.12  & 0.004  & 1301  & 4523  & 125.40  & 0.068  \\
    101   & 163   & 16.22  & 0.076  & 677   & 2203  & 84.67  & 0.077  & 1321  & 787   & 21.65  & 0.012  \\
    109   & 59    & 5.65  & 0.025  & 701   & 1259  & 47.55  & 0.042  & 1361  & 4027  & 109.16  & 0.057  \\
    113   & 67    & 6.30  & 0.027  & 709   & 379   & 14.23  & 0.012  & 1373  & 827   & 22.32  & 0.012  \\
    137   & 83    & 7.09  & 0.025  & 733   & 419   & 15.48  & 0.013  & 1381  & 691   & 18.59  & 0.010  \\
    149   & 619   & 50.71  & 0.166  & 757   & 379   & 13.77  & 0.011  & 1409  & 1619  & 43.13  & 0.022  \\
    157   & 107   & 8.54  & 0.027  & 761   & 2003  & 72.61  & 0.060  & 1429  & 739   & 19.55  & 0.010  \\
    173   & 307   & 23.34  & 0.067  & 769   & 827   & 29.82  & 0.024  & 1433  & 1907  & 50.38  & 0.025  \\
    181   & 163   & 12.12  & 0.033  & 773   & 547   & 19.67  & 0.016  & 1481  & 4019  & 104.43  & 0.051  \\
    193   & 19    & 1.37  & 0.004  & 797   & 1987  & 70.38  & 0.056  & 1489  & 883   & 22.88  & 0.011  \\
    197   & 179   & 12.75  & 0.033  & 809   & 1171  & 41.17  & 0.032  & 1493  & 947   & 24.51  & 0.012  \\
    229   & 179   & 11.83  & 0.026  & 821   & 1051  & 36.68  & 0.028  & 1549  & 787   & 20.00  & 0.009  \\
    233   & 139   & 9.11  & 0.020  & 829   & 827   & 28.72  & 0.022  & 1553  & 1427  & 36.21  & 0.017  \\
    241   & 307   & 19.78  & 0.042  & 853   & 491   & 16.81  & 0.013  & 1597  & 811   & 20.29  & 0.009  \\
    257   & 547   & 34.12  & 0.069  & 857   & 1627  & 55.58  & 0.042  & 1601  & 2707  & 67.65  & 0.031  \\
    269   & 739   & 45.06  & 0.088  & 877   & 443   & 14.96  & 0.011  & 1609  & 1571  & 39.17  & 0.018  \\
    277   & 139   & 8.35  & 0.016  & 881   & 1723  & 58.05  & 0.043  & 1613  & 2027  & 50.47  & 0.023  \\
    281   & 691   & 41.22  & 0.077  & 929   & 1579  & 51.81  & 0.036  & 1621  & 811   & 20.14  & 0.009  \\
    293   & 691   & 40.37  & 0.073  & 937   & 659   & 21.53  & 0.015  & 1637  & 1259  & 31.12  & 0.014  \\
    313   & 179   & 10.12  & 0.017  & 941   & 4603  & 150.05  & 0.104  & 1657  & 947   & 23.26  & 0.010  \\
    317   & 211   & 11.85  & 0.020  & 953   & 859   & 27.83  & 0.019  & 1669  & 971   & 23.77  & 0.011  \\
    337   & 67    & 3.65  & 0.006  & 977   & 683   & 21.85  & 0.015  & 1693  & 971   & 23.60  & 0.010  \\
    349   & 499   & 26.71  & 0.042  & 997   & 571   & 18.08  & 0.012  & 1697  & 1019  & 24.74  & 0.011  \\
    353   & 419   & 22.30  & 0.034  & 1009  & 571   & 17.98  & 0.012  & 1709  & 2179  & 52.71  & 0.023  \\
    373   & 211   & 10.93  & 0.016  & 1013  & 827   & 25.98  & 0.017  & 1721  & 4019  & 96.88  & 0.042  \\
    389   & 1051  & 53.29  & 0.076  & 1021  & 587   & 18.37  & 0.012  & 1733  & 1451  & 34.86  & 0.015  \\
    397   & 227   & 11.39  & 0.016  & 1033  & 227   & 7.06  & 0.005  & 1741  & 1019  & 24.42  & 0.011  \\
    401   & 251   & 12.53  & 0.017  & 1049  & 3011  & 92.97  & 0.059  & 1753  & 1019  & 24.34  & 0.010  \\
    409   & 331   & 16.37  & 0.022  & 1061  & 691   & 21.21  & 0.013  & 1777  & 1019  & 24.17  & 0.010  \\
    421   & 211   & 10.28  & 0.014  & 1069  & 1579  & 48.29  & 0.030  & 1789  & 907   & 21.44  & 0.009  \\
    433   & 251   & 12.06  & 0.016  & 1093  & 547   & 16.55  & 0.010  & 1801  & 859   & 20.24  & 0.008  \\
    449   & 659   & 31.10  & 0.039  & 1097  & 2371  & 71.59  & 0.044  & 1861  & 4219  & 97.80  & 0.040  \\
    457   & 83    & 3.88  & 0.005  & 1109  & 2851  & 85.61  & 0.052  & 1873  & 331   & 7.65  & 0.003  \\
    461   & 1531  & 71.31  & 0.088  & 1117  & 563   & 16.85  & 0.010  & 1877  & 1123  & 25.92  & 0.011  \\
    509   & 3923  & 173.88  & 0.198  & 1129  & 211   & 6.28  & 0.004  & 1889  & 4523  & 104.07  & 0.042  \\
    521   & 2243  & 98.27  & 0.110  & 1153  & 659   & 19.41  & 0.011  & 1901  & 3019  & 69.24  & 0.028  \\
    541   & 283   & 12.17  & 0.013  & 1181  & 4019  & 116.95  & 0.068  & 1913  & 1483  & 33.91  & 0.014  \\ \hline  \hline

    \end{tabular}%
  \label{tab:addlabel}%
\end{table}%

\small
\begin{table}[htbp]
  \centering
  \caption{The data of prime $p \equiv 3 \bmod 4$}
    \begin{tabular}{||r|r|r|r||r|r|r|r||r|r|r|r||}
    \hline
    \hline
    $p$ &  $\M(p)$ & $\frac{\M(p)}{\sqrt{p}}$ & $\frac{\M(p)}{p\log^2 p}$ & $p$ &  $\M(p)$ & $\frac{\M(p)}{\sqrt{p}}$ & $\frac{\M(p)}{p\log^2 p}$ & $p$ &  $\M(p)$ & $\frac{\M(p)}{\sqrt{p}}$ & $\frac{\M(p)}{p\log^2 p}$ \\ \hline
    7     & 3     & 1.13  & 0.113  & 563   & 1259  & 53.06  & 0.056  & 1291  & 739   & 20.57  & 0.011  \\
    11    & 3     & 0.90  & 0.047  & 571   & 179   & 7.49  & 0.008  & 1303  & 227   & 6.29  & 0.003  \\
    19    & 11    & 2.52  & 0.067  & 587   & 419   & 17.29  & 0.018  & 1307  & 2099  & 58.06  & 0.031  \\
    23    & 3     & 0.63  & 0.013  & 599   & 859   & 35.10  & 0.035  & 1319  & 1723  & 47.44  & 0.025  \\
    31    & 19    & 3.41  & 0.052  & 607   & 347   & 14.08  & 0.014  & 1327  & 211   & 5.79  & 0.003  \\
    43    & 11    & 1.68  & 0.018  & 619   & 443   & 17.81  & 0.017  & 1367  & 811   & 21.93  & 0.011  \\
    47    & 59    & 8.61  & 0.085  & 631   & 163   & 6.49  & 0.006  & 1399  & 859   & 22.97  & 0.012  \\
    59    & 307   & 39.97  & 0.313  & 643   & 379   & 14.95  & 0.014  & 1423  & 251   & 6.65  & 0.003  \\
    67    & 19    & 2.32  & 0.016  & 647   & 1163  & 45.72  & 0.043  & 1427  & 3083  & 81.61  & 0.041  \\
    71    & 43    & 5.10  & 0.033  & 659   & 907   & 35.33  & 0.033  & 1439  & 1451  & 38.25  & 0.019  \\
    79    & 19    & 2.14  & 0.013  & 683   & 467   & 17.87  & 0.016  & 1447  & 1163  & 30.57  & 0.015  \\
    83    & 131   & 14.38  & 0.081  & 691   & 419   & 15.94  & 0.014  & 1451  & 883   & 23.18  & 0.011  \\
    103   & 59    & 5.81  & 0.027  & 719   & 1459  & 54.41  & 0.047  & 1459  & 1579  & 41.34  & 0.020  \\
    107   & 83    & 8.02  & 0.036  & 727   & 419   & 15.54  & 0.013  & 1471  & 619   & 16.14  & 0.008  \\
    127   & 19    & 1.69  & 0.006  & 739   & 283   & 10.41  & 0.009  & 1483  & 1051  & 27.29  & 0.013  \\
    131   & 379   & 33.11  & 0.122  & 743   & 523   & 19.19  & 0.016  & 1487  & 2339  & 60.66  & 0.029  \\
    139   & 107   & 9.08  & 0.032  & 751   & 163   & 5.95  & 0.005  & 1499  & 1667  & 43.06  & 0.021  \\
    151   & 43    & 3.50  & 0.011  & 787   & 467   & 16.65  & 0.013  & 1511  & 1979  & 50.91  & 0.024  \\
    163   & 43    & 3.37  & 0.010  & 811   & 499   & 17.52  & 0.014  & 1523  & 907   & 23.24  & 0.011  \\
    167   & 211   & 16.33  & 0.048  & 823   & 131   & 4.57  & 0.004  & 1531  & 3907  & 99.85  & 0.047  \\
    179   & 227   & 16.97  & 0.047  & 827   & 491   & 17.07  & 0.013  & 1543  & 883   & 22.48  & 0.011  \\
    191   & 251   & 18.16  & 0.048  & 839   & 3467  & 119.69  & 0.091  & 1559  & 2531  & 64.10  & 0.030  \\
    199   & 227   & 16.09  & 0.041  & 859   & 499   & 17.03  & 0.013  & 1567  & 907   & 22.91  & 0.011  \\
    211   & 59    & 4.06  & 0.010  & 863   & 547   & 18.62  & 0.014  & 1571  & 6947  & 175.27  & 0.082  \\
    223   & 131   & 8.77  & 0.020  & 883   & 227   & 7.64  & 0.006  & 1579  & 563   & 14.17  & 0.007  \\
    227   & 139   & 9.23  & 0.021  & 887   & 971   & 32.60  & 0.024  & 1583  & 3557  & 89.40  & 0.041  \\
    239   & 571   & 36.93  & 0.080  & 907   & 227   & 7.54  & 0.005  & 1607  & 1597  & 39.84  & 0.018  \\
    251   & 947   & 59.77  & 0.124  & 911   & 1291  & 42.77  & 0.031  & 1619  & 2339  & 58.13  & 0.026  \\
    263   & 331   & 20.41  & 0.041  & 919   & 443   & 14.61  & 0.010  & 1627  & 947   & 23.48  & 0.011  \\
    271   & 179   & 10.87  & 0.021  & 947   & 563   & 18.30  & 0.013  & 1663  & 331   & 8.12  & 0.004  \\
    283   & 163   & 9.69  & 0.018  & 967   & 139   & 4.47  & 0.003  & 1667  & 2027  & 49.65  & 0.022  \\
    307   & 179   & 10.22  & 0.018  & 971   & 4051  & 130.00  & 0.088  & 1699  & 971   & 23.56  & 0.010  \\
    311   & 571   & 32.38  & 0.056  & 983   & 619   & 19.74  & 0.013  & 1723  & 443   & 10.67  & 0.005  \\
    331   & 83    & 4.56  & 0.007  & 991   & 211   & 6.70  & 0.004  & 1747  & 443   & 10.60  & 0.005  \\
    347   & 251   & 13.47  & 0.021  & 1019  & 3011  & 94.32  & 0.062  & 1759  & 691   & 16.48  & 0.007  \\
    359   & 467   & 24.65  & 0.038  & 1031  & 1907  & 59.39  & 0.038  & 1783  & 1019  & 24.13  & 0.010  \\
    367   & 211   & 11.01  & 0.016  & 1039  & 1307  & 40.55  & 0.026  & 1787  & 1163  & 27.51  & 0.012  \\
    379   & 107   & 5.50  & 0.008  & 1051  & 283   & 8.73  & 0.006  & 1811  & 4987  & 117.19  & 0.049  \\
    383   & 491   & 25.09  & 0.036  & 1063  & 883   & 27.08  & 0.017  & 1823  & 1931  & 45.23  & 0.019  \\
    419   & 1427  & 69.71  & 0.093  & 1087  & 139   & 4.22  & 0.003  & 1831  & 379   & 8.86  & 0.004  \\
    431   & 547   & 26.35  & 0.034  & 1091  & 3331  & 100.85  & 0.062  & 1847  & 2003  & 46.61  & 0.019  \\
    439   & 307   & 14.65  & 0.019  & 1103  & 947   & 28.51  & 0.017  & 1867  & 1091  & 25.25  & 0.010  \\
    443   & 331   & 15.73  & 0.020  & 1123  & 643   & 19.19  & 0.012  & 1871  & 2803  & 64.80  & 0.026  \\
    463   & 67    & 3.11  & 0.004  & 1151  & 2339  & 68.94  & 0.041  & 1879  & 2251  & 51.93  & 0.021  \\
    467   & 947   & 43.82  & 0.054  & 1163  & 691   & 20.26  & 0.012  & 1907  & 2267  & 51.91  & 0.021  \\
    479   & 787   & 35.96  & 0.043  & 1171  & 1163  & 33.99  & 0.020  & 1931  & 5347  & 121.68  & 0.048  \\
    487   & 83    & 3.76  & 0.004  & 1187  & 947   & 27.49  & 0.016  & 1951  & 1747  & 39.55  & 0.016  \\
    491   & 1187  & 53.57  & 0.063  & 1223  & 1163  & 33.26  & 0.019  & 1979  & 3571  & 80.27  & 0.031  \\
    499   & 131   & 5.86  & 0.007  & 1231  & 859   & 24.48  & 0.014  & 1987  & 1187  & 26.63  & 0.010  \\
    503   & 811   & 36.16  & 0.042  & 1259  & 3347  & 94.33  & 0.052  & 1999  & 659   & 14.74  & 0.006  \\
    523   & 331   & 14.47  & 0.016  & 1279  & 1019  & 28.49  & 0.016  &       &       &       &  \\
    547   & 139   & 5.94  & 0.006  & 1283  & 1051  & 29.34  & 0.016  &       &       &       &  \\
    \hline
    \hline
    \end{tabular}%
  \label{tab:addlabe2}%
\end{table}%

\end{document}